\documentclass[11pt]{article}
\usepackage[margin=1.05in]{geometry}
\usepackage{amsmath,amssymb,amsthm,mathtools}
\usepackage{booktabs}
\usepackage[colorlinks=true,linkcolor=blue,citecolor=blue,urlcolor=blue]{hyperref}

\newtheorem{theorem}{Theorem}
\newtheorem{proposition}[theorem]{Proposition}
\newtheorem{corollary}[theorem]{Corollary}
\newtheorem{lemma}[theorem]{Lemma}
\newtheorem{conjecture}[theorem]{Conjecture}
\theoremstyle{definition}
\newtheorem{definition}[theorem]{Definition}
\newtheorem{remark}[theorem]{Remark}
\newtheorem*{example}{Example}

\newcommand{\R}{\mathbb{R}}
\newcommand{\Z}{\mathbb{Z}}
\newcommand{\T}{\mathbb{T}}
\newcommand{\e}[1]{e^{2\pi i #1}}
\newcommand{\covol}{\operatorname{covol}}
\newcommand{\vol}{\operatorname{vol}}

\title{Linear Programming Bounds for Fibered Sphere Packings}
\author{Andrew Salmon\thanks{Independent Researcher, United States.
Email: \texttt{asalmon@alum.mit.edu}.  ORCID:
\url{https://orcid.org/0000-0002-9901-399X}.}}
\date{July 2026}

\begin{document}
\maketitle

\begin{abstract}
We study linear programming (LP) bounds for sphere packings that fiber over translates of a fixed lattice of lower rank, as well as their dual formulations. In doing so, we place the work of Conway and Sloane on fibered packings in the context of linear programming bounds for packing problems on $\R^k \times A$, where $A$ is a compact abelian group. For these programs we prove that the primal and the dual both attain their optima, so that every instance has an optimal pair satisfying complementary slackness. We study cases in which the LP bounds considered here achieve the best known sphere packing densities in dimensions $\le 9$ with prescribed translational symmetry, recovering analogues of Propositions~2, 3, 5, and 8 of Conway and Sloane, and we show that the Barnes--Wall lattice achieves the optimal sphere packing density for any $16$-dimensional packing that fibers over translates of $E_8$. In dimension $6$, we show that the natural LP bound fibering over translates of $D_4$ is in fact \emph{equivalent} to the LP bound in dimension $2$ for ordinary sphere packing, while the LP bound for $4$-dimensional packings that fiber over $A_2$ translates is implied by the LP bound in dimension $2$ but has additional rigid structure that may make it more tractable. Finally, using the discrete reduction framework of Li, we show that the linear programming bound for packings fibering over translates of $A_3$ is strictly above $1$, so the linear programming bound alone cannot prove Conjecture~4.1 of Cohn and Rajagopal.
\end{abstract}


\medskip


\section{Introduction}

The sphere packing problem asks for the densest packing of congruent, non-overlapping spheres in Euclidean space $\R^n$. In their landmark 2003 paper, Cohn and Elkies~\cite{CE03} introduced a linear programming (LP) bound that computes upper bounds on sphere packing density. This bound is sharp in dimensions $n=1$, $n=8$ (Viazovska~\cite{Via17}) and $n=24$ (Cohn et al.~\cite{CKMRV17}), where the optimal bounds are certified by exact modular form constructions called magic functions. It is also conjecturally sharp in dimension $n=2$.

However, in other dimensions $\ge 4$, the sphere packing problem is open.
To organize this complexity, Conway and Sloane~\cite{CS95} introduced the concept of \emph{fibered packings}.
They conjectured that tight packings in dimension $n$ decompose into parallel layers, each congruent to a tight packing in a lower dimension.
Subject to a sequence of postulates, they classified tight packings in dimensions $n\leq9$.

Interest in fibered packings was recently revived by Cohn and Rajagopal.
They construct new sphere packings in dimension~$5$ from valid
four-colored configurations with phases (cosets) in $D_3^*/D_3$, and
conjecture that any such colored configuration has center density at most~$1$
\cite[Conjecture~4.1]{CR24}.

In this paper, we study the linear programming bounds for fibered packings under the assumption of \emph{prescribed translation symmetry}. Specifically, we fix an $m$-dimensional lattice $L \subset \R^m$ representing the fiber, and consider packings in $\R^{k+m} = \R^k \times \R^m$ invariant under translation by $\{0\}^k \times L$. We assume $\min(L):=\min_{\ell\in L\setminus\{0\}}|\ell|\ge r$, as is necessary for each individual $L$-orbit to be $r$-separated. Index the $L$-orbits of centers by a locally finite set $I$, and write their representatives as $(b_i,t_i)$ with $b_i\in\R^k$ and $t_i\in\R^m$. (Distinct indices may have the same base coordinate.) The remaining packing conditions are
\[
  |b_i-b_j|^2+\min_{\ell\in L}|t_i-t_j-\ell|^2\ge r^2
  \qquad(i\ne j).
\]
Equivalently, the points $(b_i,\bar t_i)$ form an $r$-separated
configuration (with multiplicity in the base projection) on the flat cylinder
\[
  X = \R^k \times \T, \qquad \T = \R^m/L,
\]
with the quotient metric
\[
  d_X((x,\bar y), (x',\bar y'))^2 = |x-x'|^2 + \min_{\ell \in L} |y-y'-\ell|^2.
\]
Let $\delta_B$ denote the number of $L$-orbits of centers per unit
$k$-dimensional base volume, counting multiplicity when several orbits have
the same base coordinate. Then the ambient center density is
$\delta_B/\covol(L)$, where $\covol(L)$ is the volume of a fundamental
domain of $L$, and the sphere-packing density in $\R^{k+m}$ is
\[
  \Delta(P) = \delta_B \frac{\vol B^{k+m}_{r/2}}{\covol(L)}.
\]
Here $B^n_s=\{x\in\R^n:|x|\le s\}$ denotes the closed Euclidean ball of
radius $s$ in dimension $n$.
Throughout, $r$ is a center-to-center exclusion distance (so the packed
spheres have radius $r/2$), while every LP value and every ``base density''
is a center density in the Euclidean base. Under dilation of the base by a
factor $a$, exclusion distances are multiplied by $a$ and base densities
by $a^{-k}$. Unless explicitly stated otherwise, root lattices use the
standard normalization in which their minimal squared norm is $2$.
We write $BW_{16}$ for the $16$-dimensional Barnes--Wall lattice.
Thus, the density problem for aligned fibered packings is exactly the density problem for separated configurations on the flat cylinder $X$.

The natural linear programming bound for packing problems of this type takes as its base space $X=\R^k\times A$, where $A$ is compact abelian. When $A$ is a torus, this recovers the flat cylinder case above and gives upper bounds for fibered packing problems with prescribed translational symmetry, while when $A$ is discrete, it yields upper bounds for colored packing problems. We prove functoriality under homomorphisms of $A$, and we prove that the primal and the dual program both attain their optima, so that an optimal pair satisfies complementary slackness. Strong duality for these programs is a specialization of the duality theorem of Berdysheva, Farkas, Ga\'al, Ramabulana, and R\'ev\'esz~\cite{BFGRR26} for the Delsarte extremal problem on locally compact abelian groups; the primal attainment argument adapts~\cite[Proposition~A.1]{CdLS22}, and the dual attainment theorem generalizes~\cite[Proposition~3.6]{CdLS22}.

For a one-dimensional base ($k=1$) and fiber dimensions $m\leq7$, we show that the cylinder LP is certified sharp by a simple triangular function. These cases traverse lattices achieving the best known sphere-packing densities in the corresponding dimensions and recover the $k=1$ aligned classification of Conway and Sloane unconditionally. Separately, for the $D_8$ fiber, we characterize a colored packing problem in dimension 9 by a four-color kernel built from the same triangle profile in two color-difference channels, for which all members of the fluid diamond family give a dual certificate.

We give two further applications of sphere-packing magic functions. First, we apply the magic functions on $\R^8$ and $\R^{24}$ to cylinder problems obtained as quotients of those spaces. Second, we use Viazovska's $E_8$ magic function to construct a sharp cylinder magic function for the Barnes--Wall lattice viewed as a fibered packing over $E_8$.

For the $E_6$ over $D_4$ problem, we prove that flexibility in the optimal coloring patterns forces all nontrivial character channels of a two-point certificate to vanish identically, so that the two-point colored or cylinder LP is sharp if and only if the ordinary planar Cohn--Elkies LP is sharp.

Following Li's discrete reduction framework~\cite{Li25}, we restrict the
Euclidean base to $(\Z/m\Z)^k$ while retaining the finite color group. An
exact rational dual certificate gives a rigorous lower bound for the
continuous translation-invariant colored LP. For the $D_3$-fibered problem
of Cohn and Rajagopal (written $D_5/A_3$, with $A_3\cong D_3$), this lower
bound is strictly greater than the conjectured density, so this two-point
LP cannot prove \cite[Conjecture~4.1]{CR24}.

\begin{table}[htbp]
\centering
\footnotesize
\setlength{\tabcolsep}{4.5pt}
\caption{Proved and conjecturally sharp cylinder or colored LP base
center-density values.
Entries marked conditional follow from sharpness of the planar
Cohn--Elkies LP.}
\label{tab:exact-values}
\vspace{0.5em}
\begin{tabular}{ccccc}
\toprule
Space $\Lambda$ & Base $k$ & Fiber $L$ & Target LP value & Status/reference \\
\midrule
\multicolumn{5}{l}{\textbf{One-dimensional base}} \\
$A_2$ & $1$ & $A_1$ & $\sqrt{2/3}$ & Thm.~\ref{thm:triangle} \\
$A_3$ & $1$ & $A_2$ & $\sqrt{3}/2$ & Thm.~\ref{thm:triangle} \\
$D_4$ & $1$ & $D_3$ & $1$ & Thm.~\ref{thm:triangle} \\
$D_5$ & $1$ & $D_4$ & $1$ & Thm.~\ref{thm:triangle} \\
$E_6$ & $1$ & $D_5$ & $2/\sqrt{3}$ & Thm.~\ref{thm:triangle} \\
$E_7$ & $1$ & $E_6$ & $\sqrt{3/2}$ & Thm.~\ref{thm:triangle} \\
$E_8$ & $1$ & $E_7$ & $\sqrt{2}$ & Thm.~\ref{thm:triangle} \\
$E_8$ & $1$ & $D_7$ & $2$ & Thm.~\ref{thm:triangle} \\
$\Lambda_{17}$ & $1$ & $BW_{16}$ & $1$ & Thm.~\ref{thm:triangle} \\
\midrule
\multicolumn{5}{l}{\textbf{Higher-dimensional base}} \\
$E_8$ & $2$ & $E_6$ & $\sqrt{3}$ & Thm.~\ref{thm:sharp_fibrations} \\
$E_8$ & $4$ & $D_4$ & $2$ & Thm.~\ref{thm:sharp_fibrations} \\
$BW_{16}$ & $8$ & $E_8$ & $1$ & Thm.~\ref{thm:bw16-torus} \\
\midrule
\multicolumn{5}{l}{\textbf{Planar-base problems controlled by the planar LP}} \\
$E_6$ & $2$ & $D_4$ & $2/\sqrt{3}$ &
conditional; Cor.~\ref{cor:e6d4cyl} \\
$D_4$ & $2$ & $A_2$ & $\sqrt{3}/2$ &
Conj.~\ref{conj:d4a2}; Prop.~\ref{prop:d4a2-planar} \\
\bottomrule
\end{tabular}
\end{table}

\paragraph{Outline of the paper.} In Section~\ref{sec:formulations}, we define the common primal and dual programs on $X=\R^k\times A$ and specialize them to cylinders and finite color groups. Section~\ref{sec:solved} gives the exactly solved cases: the $k=1$ cylinder certificates for fiber dimensions $m\leq7$, the finite-color $D_8$ certificate, cylinder quotients of $\R^8$ and $\R^{24}$, and a sharp cylinder magic function for the Barnes--Wall lattice viewed as a fibered packing over $E_8$. Section~\ref{sec:collapse} collects the conjectural cases and their relations to the two-dimensional sphere packing LP, including the $E_6$ over $D_4$ collapse theorem. Section~\ref{sec:reduction} then develops discrete reduction duals and proves the certified $D_5/A_3$ gap. Section~\ref{sec:duality} proves the duality theorems used throughout: strong duality, attainment of the primal and of the dual, and the complementary slackness conditions that the certificates rely on.

\paragraph{Use of artificial intelligence.} Artificial-intelligence tools
assisted with drafting and numerical computation.  Any errors are the
author's own.

\section{The Common Linear Program on $\R^k\times A$}\label{sec:formulations}

Both the cylinder and the translation-invariant colored bounds are instances
of one linear program on a locally compact abelian group.  Let $A$ be a
compact abelian group with normalized Haar measure and Pontryagin dual
$\widehat A$, and put
\[
 X=\R^k\times A,
 \qquad
 \widehat X=\R^k\times\widehat A.
\]
For $\chi\in\widehat A$, write
\[
 f_\chi(x)=\int_A f(x,a)\overline{\chi(a)}\,dm_A(a),
 \qquad
 \widehat f(\xi,\chi)=\widehat{f_\chi}(\xi).
\]
Thus the compact factor contributes discrete Fourier channels, whether it
is a torus or a finite group.  In the packing applications below, these are
the continuous- and discrete-phase descriptions of a lattice fiber; the
abstract group $A$ is not meant to encode a non-lattice fiber.

\subsection{The linear program and its dual}

All functions below are real-valued, continuous, and integrable with
respect to Haar measure.

Let $K(X)$ be the cone of continuous, integrable positive-definite
functions on $X$.  By Bochner's theorem, this is equivalent to
$\widehat f(\xi,\chi)\geq0$ for every $(\xi,\chi)\in\widehat X$.

\begin{definition}[$C$-constrained linear program]
\label{def:c-constrained-lp}
Let $C\subset X=\R^k\times A$ be closed with $0\notin C$.  The
$C$-constrained linear program on $X$ has value
\begin{equation}\label{eq:primal_lca_val}
 P_X(C)=\inf\left\{f(0):
 \begin{array}{l}
  f\in K(X),\\
  f(p)\leq0\quad\text{for every }p\in C,\\
  \displaystyle\int_X f(p)\,dm_X(p)=1
 \end{array}
 \right\}.
\end{equation}
Here $m_X$ is the product Haar measure. The difference set $C$ encoding the geometry enters only as a restriction on which the auxiliary function is required
to be nonpositive.
\end{definition}

\begin{example}[Torus]
Let $L\subset\R^m$ be a lattice and take
\[
 A=\T=\R^m/L,
 \qquad
 X=\R^k\times\T.
\]
Then $\widehat A=L^*$, and the character channels are precisely the fiber
Fourier modes.  For sphere packings with exclusion distance $r$, the set of
admissible nonzero differences is
\[
 C=\{(x,\bar y)\in X:d_X((x,\bar y),0)\geq r\}.
\]
This is the cylinder linear program for aligned $L$-fibered packings.
\end{example}

\begin{example}[Finite abelian group]
Let $A=G$ be a finite abelian group with normalized counting measure.  If
$C_\delta\subset\R^k$ is the set of admissible nonzero differences for a
pair of colors whose difference is $\delta\in G$, then take
\[
 C=\bigcup_{\delta\in G}(C_\delta\times\{\delta\})
 \subset\R^k\times G.
\]
When each $C_\delta$ is closed and $0\notin C_0$, this is a closed forbidden
set not containing the identity, and the resulting $C$-constrained program
is the translation-invariant finite-color linear program.
\end{example}

\begin{theorem}[Primal attainment]\label{thm:lca-primal-attainment}
Let $U\subset X$ be an open neighborhood of the identity with
$m_X(U)<\infty$, and set $C=X\setminus U$.  Assume in addition that
$C=\overline{\operatorname{int}C}$.  Then the infimum in
\eqref{eq:primal_lca_val} is attained by a real-valued function
\[
 f\in K(X),\qquad
 f\leq0\ \hbox{on }C,\qquad \int_X f\,dm_X=1.
\]
\end{theorem}

\begin{proof}
See Section~\ref{app:primal-attainment}.
\end{proof}

We now formulate the dual of this problem.  Let $M_{\mathrm{temp}}(X)$ be
the cone of nonnegative tempered Radon measures on $X$.  The dual linear
program is
\begin{equation}\label{eq:dual_lca_val}
 D_X(C)=\sup\left\{
 \widehat\mu\big(\{(0,\mathbf1)\}\big):
 \begin{array}{l}
   \mu\in M_{\mathrm{temp}}(X),\quad
   \operatorname{supp}\mu\subseteq\{0\}\cup C,\\
   \mu(\{0\})=1,\quad
   \widehat\mu\text{ is a nonnegative tempered Radon measure}
 \end{array}
 \right\}.
\end{equation}
Here $\mathbf1\in\widehat A$ is the trivial character.  Thus the primal is
normalized by $\int_Xf\,dm_X=1$, while the dual is normalized by
$\mu(\{0\})=1$.

Indeed, put $z=\widehat\mu(\{(0,\mathbf1)\})$.  Since $0\notin C$, every
dual feasible measure decomposes uniquely as $\mu=\delta_0+\nu$, where
$\nu$ is nonnegative and supported on $C$.  Removing the atom of
$\widehat\mu$ at $(0,\mathbf1)$ gives the equivalent condition
\[
 T:=\mu-zm_X=\delta_0+\nu-zm_X\in K^*(X),
 \qquad
 \widehat T=\widehat\mu-z\delta_{(0,\mathbf1)}\geq0.
\]
Here $K^*(X)$ denotes the cone of positive-definite tempered
distributions on $X$.

Weak duality follows by pairing measures with continuous
functions.  If $f$ is primal feasible and $\mu$ is dual feasible, then
formally
\[
 f(0)-z=\langle T,f\rangle-\int_C f\,d\mu\geq0,
\]
because $\langle T,f\rangle\geq0$ and $f\leq0$ on the support of the
nonnegative measure $\mu|_C$.  Thus $P_X(C)\geq D_X(C)$.  Both steps
need justification, since a primal feasible $f$ is only continuous and
integrable: the inequality $\langle T,f\rangle\geq0$ rests on the
identity $\langle T,f\rangle=\langle\widehat T,\widehat f\rangle$, which
holds for Schwartz functions but not for every $f\in K(X)$, and
$\int_Cf\,d\mu$ need not converge.
Proposition~\ref{prop:app-mollify} settles both by mollifying $f$ into
the Schwartz class, where the pairing is legitimate, and passing to the
limit.

\begin{theorem}[Strong duality]\label{thm:strong-duality-lca}
For every closed subset $C\subset X=\R^k\times A$ with $0\notin C$, the
common linear program has no duality gap:
\[
  P_X(C) = D_X(C).
\]
\end{theorem}

\begin{proof}
The absence of a duality gap is a specialization of
\cite[Theorem~5.3 and Remark~5.4]{BFGRR26}; see
Section~\ref{sec:duality}.
\end{proof}

\begin{theorem}[Dual attainability]\label{thm:dual-attainability}
For every closed subset $C\subset X=\R^k\times A$ with $0\notin C$, the
supremum defining $D_X(C)$ is attained: there is a measure $\mu$
feasible for \eqref{eq:dual_lca_val} whose objective value
$\widehat\mu(\{(0,\mathbf1)\})$ equals $D_X(C)$.
\end{theorem}

\begin{proof}
See Section~\ref{app:attainability}.
\end{proof}

\begin{remark}
For $X=\R^n$ this is \cite[Proposition~3.6]{CdLS22}, where the dual
optimum is attained by a tempered distribution and the compactness comes
from the Banach--Alaoglu theorem applied to a polar set in
$\mathcal{S}'(\R^n)$.  The conclusion there is equivalent to ours in that
case: the dual constraints force such a distribution to be
$\delta_0+\nu$ with $\nu$ a nonnegative measure, and conversely the
limit produced in Section~\ref{app:attainability} is again tempered,
since a positive definite measure is translation bounded
\cite[Section~4]{AGdL} and translation bounded measures on
$\R^k\times A$ have polynomially bounded growth.
\end{remark}

When $C=-C$, let $G_C$ be the graph on $X$ in which two distinct elements
$g$ and $g'$ are adjacent if $g-g'\notin C$.  Thus independent sets in
$G_C$ are precisely sets whose nonzero differences lie in $C$.  The linear
program above is the translation-invariant analogue of the $\vartheta'$
bound for $G_C$; compare the positive-semidefinite-measure formulation in
\cite[Proposition~5.1]{CS21}.  From this viewpoint, the functoriality below
is the linear-programming counterpart of the monotonicity of theta-type
bounds under graph homomorphisms~\cite{Lov79}.  With $k$ fixed, write
$\operatorname{LP}_A(C)=P_{\R^k\times A}(C)$.

\begin{proposition}[Functoriality]\label{prop:lca-functoriality}
Let $A$ and $B$ be compact abelian groups, let $\phi:A\to B$ be a continuous
homomorphism, and write
$\Phi=\operatorname{id}_{\R^k}\times\phi$.  Suppose that the closed
constraint sets $C_A\subset\R^k\times A$ and
$C_B\subset\R^k\times B$ do not contain the identity and satisfy
\[
 \Phi(C_A)\subseteq C_B.
\]
Then
\[
 \operatorname{LP}_A(C_A)\leq \operatorname{LP}_B(C_B).
\]
\end{proposition}

\begin{proof}
Set $H=\phi(A)$, a closed subgroup of $B$, and let $m_H$ denote its
normalized Haar measure, viewed as a measure on $B$.  We prove the
inequality on the dual side.  Let $\mu_A$ be dual feasible on
$\R^k\times A$, put
$z=\widehat\mu_A(\{(0,\mathbf1)\})$, and set
\[
 T_A=\mu_A-zm_{\R^k\times A}\in K^*(\R^k\times A).
\]
Pushforward by a homomorphism preserves positive definiteness, and
$\Phi_*m_{\R^k\times A}=m_{\R^k}\otimes m_H$.  Thus
\[
 \Phi_*T_A
 =\Phi_*\mu_A-z(m_{\R^k}\otimes m_H)
 \in K^*(\R^k\times B).
\]
Moreover,
\[
 m_{\R^k}\otimes(m_H-m_B)
\]
is positive definite: its Fourier transform is the nonnegative measure
supported at $\xi=0$ whose weight at $\chi\in\widehat B$ is
\[
 \mathbf 1_{H^\perp}(\chi)-\mathbf 1_{\{\mathbf1\}}(\chi).
\]
Consequently
\[
 \Phi_*\mu_A-zm_{\R^k\times B}
 =\Phi_*T_A+z\,m_{\R^k}\otimes(m_H-m_B)
 \in K^*(\R^k\times B).
\]
The pushed-forward measure is nonnegative and is supported on
$\{0\}\cup\Phi(C_A)\subseteq\{0\}\cup C_B$.  Moreover, no point of
$C_A$ maps to $0$, since $0\notin C_B$, so
$(\Phi_*\mu_A)(\{0\})=1$.  The last display shows that the Fourier atom
of $\Phi_*\mu_A$ at $(0,\mathbf1)$ is at least $z$.  Hence
$D_{\R^k\times A}(C_A)\leq D_{\R^k\times B}(C_B)$, and
Theorem~\ref{thm:strong-duality-lca} gives the stated primal inequality.
\end{proof}

In particular, Proposition~\ref{prop:lca-functoriality} applies when a
finite phase group $G$ is included in a torus $\R^m/L$.  Whenever the
admissible differences for $G$ map into those for the torus, the discrete
phase bound is at most the full cylinder bound.

\begin{proposition}[Complementary slackness]
\label{prop:lca-complementary-slackness}
Let $f$ be primal optimal and let $\mu$ be dual optimal; by
Theorems~\ref{thm:lca-primal-attainment}
and~\ref{thm:dual-attainability} such a pair exists whenever
$C=X\setminus U$ with $m_X(U)<\infty$ and
$C=\overline{\operatorname{int}C}$.  Set
\[
 \nu=\widehat\mu-D_X(C)\,\delta_{(0,\mathbf1)}\geq0 .
\]
Then $f$ is $\mu$-integrable,
$\widehat f$ is $\nu$-integrable, and
\[
 \int_C f\,d\mu=0
 \qquad\hbox{and}\qquad
 \int_{\widehat X}\widehat f\,d\nu=0 ;
\]
that is, $f=0$ $\mu|_C$-almost everywhere and $\widehat f=0$
$\nu$-almost everywhere.  When $\mu|_C$ and $\nu$ are atomic, $f$
vanishes at every point of $\operatorname{supp}\mu|_C$ and
$\widehat f$ at every point of $\operatorname{supp}\nu$.
\end{proposition}

\begin{proof}
Theorem~\ref{thm:strong-duality-lca} gives
$f(0)=P_X(C)=D_X(C)=\widehat\mu(\{(0,\mathbf1)\})$.
Proposition~\ref{prop:app-mollify} applies with $\mu$ in the role of
its $T$ and $\mu|_C$ in the role of its $\mu$: dual feasibility gives
$\mu=\delta_0+\mu|_C$ with $\mu|_C\geq0$ supported on $C$, and
$\widehat\mu=D_X(C)\delta_{(0,\mathbf1)}+\nu$.
\end{proof}

\subsection{Packings as dual certificates}\label{sec:config-dual}

Periodic packing configurations produce dual-feasible measures through their two-point
autocorrelations, and arbitrary packings are approximated by
periodic ones with increasingly large periods.  In this way, the dual linear program
can be viewed as a relaxation of the geometric packing problem.

Call a configuration $P\subset X$ \emph{admissible for $C$} if all its
nonzero differences lie in $C$.  Let
$Q_R=[-R/2,R/2)^k$ and define its upper base density by
\[
 \delta_C(P)=\limsup_{R\to\infty}\ \sup_{t\in\R^k}
 \frac{\#\big(P\cap((t+Q_R)\times A)\big)}{R^k}.
\]
For the density statements below, suppose the forbidden differences have
bounded base projection: there is $R_C<\infty$ such that
\[
 \{(x,a)\in X:|x|\geq R_C\}\subseteq C.
\]
This condition holds in every packing application in this paper.
Call $P$ \emph{$\Lambda$-periodic}, for a
full-rank lattice $\Lambda\subset\R^k$, if $P$ is invariant under
$\Lambda\times\{0\}$ and meets a fundamental domain in finitely many
points.  If such a $P$ has $N$ orbits, then
$\delta_C(P)=N/\covol(\Lambda)$.  In the cylinder and colored packing
problems, $\delta_C(P)$ is the base density $\delta_B$ of the introduction.

\begin{lemma}[Autocorrelation certificates]\label{lem:config-dual}
Let $P$ be a nonempty $\Lambda$-periodic configuration admissible for $C$,
with density $\delta=\delta_C(P)$ and orbit representatives
$p_1,\dots,p_N$.  Its volume-averaged autocorrelation
\[
 \gamma=\frac1{\covol(\Lambda)}
 \sum_{i,j=1}^N\sum_{\lambda\in\Lambda}\delta_{p_i-p_j+(\lambda,0)}
\]
defines the dual feasible measure $\alpha=\gamma/\delta$, whose objective
value is $\widehat\alpha(\{(0,\mathbf1)\})=\delta$.
In particular $D_X(C)\geq\delta_C(P)$ for every such configuration.
\end{lemma}

\begin{proof}
The diagonal terms ($i=j$, $\lambda=0$) contribute
$(N/\covol(\Lambda))\,\delta_0=\delta\,\delta_0$.  Every other supporting
point is a difference of two distinct points of $P$ and therefore lies in
$C$.  Periodicity gives polynomially bounded local point counts, so
$\gamma$ is a tempered nonnegative measure and
$\alpha:=\gamma/\delta$ is supported on $\{0\}\cup C$ with
$\alpha(\{0\})=1$.

We compute the Fourier transform by Poisson summation for the cocompact
subgroup $\Lambda\times\{0\}$, valid with absolute convergence for
Bruhat--Schwartz functions: writing $q=(x,a)$ and
$p_i=(b_i,a_i)$,
\[
 \sum_{\lambda\in\Lambda}\psi(q+(\lambda,0))
 =\frac1{\covol(\Lambda)}
  \sum_{(\xi,\chi)\in\Lambda^*\times\widehat A}
  \widehat\psi(\xi,\chi)\,e^{2\pi i\langle\xi,x\rangle}\chi(a)
 \qquad(\psi\in\mathcal S(X)).
\]
Summing over the $N^2$ pairs gives
\[
 \langle\gamma,\psi\rangle
 =\frac1{\covol(\Lambda)^2}
  \sum_{(\xi,\chi)\in\Lambda^*\times\widehat A}
  |S(\xi,\chi)|^2\,\widehat\psi(\xi,\chi),
 \qquad
 S(\xi,\chi)=\sum_{i=1}^N e^{2\pi i\langle\xi,b_i\rangle}\chi(a_i).
\]
Hence
\[
 \widehat\gamma
 =\frac1{\covol(\Lambda)^2}
  \sum_{(\xi,\chi)\in\Lambda^*\times\widehat A}
  |S(\xi,\chi)|^2\,\delta_{(\xi,\chi)},
\]
a nonnegative atomic measure whose atom at the trivial character
$(0,\mathbf1)$ equals $N^2/\covol(\Lambda)^2=\delta^2$.  Hence
$\widehat\alpha=\delta^{-1}\widehat\gamma$ is nonnegative and has atom
$\delta$ at $(0,\mathbf1)$.  Thus $\alpha$ is dual feasible with value
$\delta$.
\end{proof}

\begin{corollary}[Density bound]\label{cor:density-bound}
Every configuration $P$ admissible for $C$ satisfies
\[
 \delta_C(P)\leq P_X(C).
\]
\end{corollary}

\begin{proof}
The periodic case follows from Lemma~\ref{lem:config-dual} and strong
duality.  For arbitrary $P$, fix a translate $t+Q_R$ and repeat the
finite patch $P\cap((t+Q_R)\times A)$ with period
$T\Z^k$, where $T=R+2R_C$.  Differences within one copy lie in $C$.
Between distinct copies, some base coordinate has absolute value at least
$T-R=2R_C$, so those differences lie in $C$ as well.  The resulting
periodic configuration has density
\[
 \frac{\#\big(P\cap((t+Q_R)\times A)\big)}{T^k}
 \leq P_X(C).
\]
Taking the supremum over $t$, multiplying by $(T/R)^k$, and then letting
$R\to\infty$ proves the claim.
\end{proof}

\begin{remark}[Optimal configurations and slackness]
\label{rem:config-slackness}
Suppose $P$ is a periodic configuration admissible for $C$, while $f$
satisfies the primal sign and positivity conditions and
$\widehat f(0,\mathbf1)>0$.  If
\[
 \delta_C(P)=\frac{f(0)}{\widehat f(0,\mathbf1)}=P_X(C),
\]
then the autocorrelation certificate of Lemma~\ref{lem:config-dual} and the
normalization $f/\widehat f(0,\mathbf1)$ are dual and primal optimal.  The
mollification argument in
Proposition~\ref{prop:app-mollify} yields the pointwise complementary
slackness conditions
\[
 f(p-p')=0\quad(p\ne p'\in P),
 \qquad
 \widehat f=0\ \text{on}\ \operatorname{supp}\widehat\gamma
 \setminus\{(0,\mathbf1)\},
\]
that is, $f$ vanishes on the difference set of $P$ and $\widehat f$
vanishes on its Bragg spectrum (the support of its Fourier-transformed
autocorrelation) away from the trivial character.  These are
the conditions exploited in Sections~\ref{sec:d8fluid-short}
and~\ref{sec:e6d4_collapse}.  Corollary~\ref{cor:density-bound} applies
to arbitrary configurations, while this pointwise slackness statement
concerns periodic optimizers.
\end{remark}

\section{Exactly Solved Cases}\label{sec:solved}

\subsection{One-Dimensional Aligned Packings ($k=1$)}\label{sec:one_dim}

Let $R(L)$ be the covering radius of the fiber lattice $L \subset \R^m$ and set
\[
  g = \sqrt{r^2 - R(L)^2}.
\]
We assume $R(L) < r$, so $g > 0$.

Call a phase $\bar y\in\R^m/L$ a \emph{deep hole} if its distance from
$L$ is $R(L)$.  The \emph{deep-hole graph} has phase classes as vertices,
with an edge between two phases when their difference is a deep hole.

\begin{theorem}[Triangle Certificate]\label{thm:triangle}
For base dimension $k=1$, the function
\[
  f_\triangle(x,\bar y) = \left(1-\frac{|x|}{g}\right)^+
\]
is feasible for the cylinder LP and certifies the base density bound $\delta_B \le 1/g$.
\end{theorem}

\begin{proof}
Since every point in the torus $\T$ is within distance $R(L)$ of $0$, the condition $d_X((x,\bar y), 0) \ge r$ implies $|x|^2 + R(L)^2 \ge r^2$, which simplifies to $|x| \ge g$. Consequently, $f_\triangle$ vanishes on the region $d_X(p,0) \ge r$.

Since $f_\triangle(x, \bar y)$ is independent of $y$, its only nonzero fiber Fourier frequency is $w=0$. The $w=0$ channel is the ordinary one-dimensional Fourier transform:
\[
  \widehat{f_0}(\xi) = g \left(\frac{\sin(\pi g\xi)}{\pi g\xi}\right)^2 \ge 0.
\]
All other channels $\widehat{f_w}(\xi)$ for $w \ne 0$ vanish identically. Normalizing by $\widehat{f_0}(0) = g$, the certified value is $f(0,0)/\widehat{f_0}(0) = 1/g$.

Thus $P_X(C)\leq1/g$, while Corollary~\ref{cor:density-bound} bounds the
base density of every aligned $L$-fibered packing by $P_X(C)$
and hence by $1/g$.
\end{proof}

\begin{theorem}[Rigidity for $k=1$]\label{thm:rigidity}
Let $r \ge 2R(L)/\sqrt{3}$. Every periodic $r$-separated aligned $L$-fibered packing attaining the optimal base density $\delta_B = 1/g$ has base set $B$ equal to an arithmetic progression of gap $g$, and every pair of consecutive fibers has phase difference a deep hole of $L$. Consequently, the optimal aligned $L$-fibered packings are exactly the bi-infinite walks on the deep-hole graph of $L$, up to isometry.
\end{theorem}

\begin{proof}
Two fibers over the same base point are at cylinder distance $\le R(L) < r$, so base points must be distinct. A consecutive base gap $x$ between fibers with phase difference $\Delta t$ satisfies $x^2 + \min_\ell |\Delta t - \ell|^2 \ge r^2$. Since $\min_\ell |\Delta t - \ell| \le R(L)$, we have $x^2 \ge r^2 - R(L)^2 = g^2$, meaning all gaps are at least $g$. In a period of length $T$ containing $N$ points, $T \ge N g$, so attaining the bound $\delta_B = 1/g$ forces every consecutive gap to equal $g$.

Equality in the gap forces $\min_\ell |\Delta t - \ell| \ge \sqrt{r^2 - g^2} = R(L)$. Since $R(L)$ is the maximum possible covering distance, $\Delta t$ must be a deep hole of $L$. Non-consecutive layers are separated in the base by at least $2g$, and since $r \ge 2R(L)/\sqrt{3}$, we have $2g \ge r$. Thus, non-consecutive layers satisfy the packing condition for any phase choices.
\end{proof}

This recovers the aligned $k=1$ content of the Conway--Sloane laminated tower. For example, $A_2$ over $A_1$ is the unique alternating deep-hole walk; Barlow packings arise from $A_2$ deep-hole walks; and the $\Lambda_5^i$ family arises from walks on the deep-hole graph of $D_4$. This also extends to higher-dimensional fibers: for example, laminating the 17-dimensional lattice $\Lambda_{17}$ over the 16-dimensional Barnes--Wall lattice $BW_{16}$ (with covering radius $R(BW_{16}) = 1$ at packing distance $r=\sqrt{2}$) yields base gap $g = \sqrt{2 - 1^2} = 1$, and its base density is certified sharp by the triangle function.

For reference, the covering-radius and gap data used for the
root-lattice entries in the one-dimensional part of
Table~\ref{tab:exact-values} are as follows (in every case $r^2=2$):
\[
\begin{array}{c|cccccccc}
L&A_1&A_2&D_3&D_4&D_5&E_6&E_7&D_7\\ \hline
R(L)^2&\frac12&\frac23&1&1&\frac54&\frac43&\frac32&
\frac74\\
g^2=r^2-R(L)^2&\frac32&\frac43&1&1&\frac34&\frac23&
\frac12&\frac14
\end{array}
\]
Thus the corresponding values $1/g$ are exactly the first eight entries
of Table~\ref{tab:exact-values}.

\subsection{Sharpness of Fluid Diamonds for the Colored $D_8$-Fibered Problem in Dimension 9}
\label{sec:d8fluid-short}

The scalar triangle theorem above assumes $R(L)<r$.  For $L=D_8$ at
$r=\sqrt2$ one has $R(D_8)=r$, so $g=0$ and every test function constant
in the fiber phase fails to give a bound.  Nevertheless,
the problem is exactly solvable when phases are restricted to
\[
 C=D_8^*/D_8\cong(\Z/2)^2=\{0,1,g,g'\},\qquad g+g'=1.
\]
The four color-difference thresholds are
\[
 d_0=\sqrt2,\qquad d_1=1,\qquad d_g=d_{g'}=0.
\]

For $s\in[0,1)$, define the period-two colored configuration
\begin{equation}\label{eq:fluid-family}
 P_s=\{(2n,0),(2n+s,g),(2n+1,1),(2n+1+s,g'):
 n\in\Z\}\subset\R\times C.
\end{equation}
Set $w_s=(1/2,\ldots,1/2,s)\in\R^9$.  These are the aligned descriptions
of the classical fluid diamonds
$D_9\cup(D_9+w_s)$ from~\cite[Section~9C and Lemma~9C]{CS95}.  The parameter $s$ slides
one diamond subpacking relative to the other while leaving the four
$D_8$ cosets fixed.  At
$s=0$ pairs of layers merge into $E_8$ layers, giving the
$\Lambda_9$-type endpoint, while $s=1/2$ gives the lattice $D_9^+$.

\begin{theorem}[Discrete fluid-diamond LP]\label{thm:d8fluidlp-short}
The translation-invariant four-color LP for aligned $D_8$-fibered
packings of $\R^9$ is sharp at base density $2$.  Every $P_s$ in
\eqref{eq:fluid-family} attains this value and supplies a dual-optimal
autocorrelation measure.  On the primal side, with
\[
 T(x)=(1-|x|)^+,
\]
an explicit magic kernel is
\[
 \varphi_0=\varphi_1=T,
 \qquad \varphi_g=\varphi_{g'}=0,
 \qquad F_{cc'}(x)=\varphi_{c-c'}(x).
\]
Consequently the entire one-parameter discrete-phase fluid-diamond family
has sphere-packing density $\sqrt2\,\pi^4/945$, equal to that of
$\Lambda_9$.
\end{theorem}

\begin{proof}
The space-side conditions are immediate: $T$ vanishes for $|x|\ge1$,
while the two kernels whose sign is required on the whole line vanish
identically.  Writing a character as
$\chi(g)=\alpha$, $\chi(1)=\beta$ with $\alpha,\beta\in\{\pm1\}$, its
Fourier channel is
\[
 f_{\alpha,\beta}
 =\varphi_0+\alpha\varphi_g+\beta\varphi_1
   +\alpha\beta\varphi_{g'}=(1+\beta)T.
\]
Thus the two $\beta=+1$ channels equal $2T$ and have nonnegative Fourier
transform
\[
 2\widehat T(\xi)
 =2\left(\frac{\sin\pi\xi}{\pi\xi}\right)^2,
\]
whereas the two $\beta=-1$ channels vanish.  Since
$F_{cc}(0)=1$ and the trivial Fourier eigenvalue at zero is $2$, the
colored LP gives $\delta_B\le2$.

Every member of the fluid family is feasible for this four-color problem.
Consecutive differences alternate between $g$ and $g'$, whose threshold is
zero; layers one unit apart have color difference $1$, whose threshold is
one; and equal colors recur only after two units.  Thus $P_s$ is a packing
for every $s$, and its four points per period of length two give base
density $2$.

We can see sharpness directly on the dual side.  Let
\[
 \omega_s=\sum_{p\in P_s}\delta_p
\]
and let $\gamma_s$ be its volume-averaged autocorrelation on
$\R\times C$.  Its coefficient at the identity is the density $2$, so
\[
 \alpha_s:=\frac12\gamma_s=\delta_0+\mu_s
\]
with $\mu_s\geq0$ supported on the nonzero differences of $P_s$, hence on
the admissible-difference set of the \emph{four-color} problem---a more
stringent support condition than dual feasibility for the cylinder LP.
This is Lemma~\ref{lem:config-dual} applied on $X=\R\times C$: the
diffraction measure $\widehat\gamma_s$ is nonnegative, and
\[
 \widehat\gamma_s(\{(0,\mathbf1)\})=4,
\]
the square of the density, where $\mathbf1\in\widehat C$ is the trivial
character.
Therefore
\[
 \alpha_s-2m_{\R\times C}
 =\delta_0+\mu_s-2m_{\R\times C}\in K^*(\R\times C).
\]
Thus $\alpha_s$ is dual feasible for the four-color program for every
$s$.  The primal kernel and these dual autocorrelations have common value
$2$, proving optimality on both sides; by
Remark~\ref{rem:config-slackness}, each $\alpha_s$ is dual optimal.
\end{proof}

This certificate uses the triangle profile but is not the scalar
triangle-cylinder certificate: it is phase-dependent, with two nonzero
color-difference kernels (equivalently one nonzero character profile,
appearing in two characters) and two identically zero ones.  What remains
open is the extension from the four cosets $D_8^*/D_8$ to arbitrary phases
in $\R^8/D_8$, as well as independently rotated layers.

\subsection{Periodization from Solved Total Spaces}\label{sec:periodization}

The cylinder LP is also sharp whenever it is obtained by slicing a lattice whose ordinary Euclidean Cohn--Elkies problem is solved.

\begin{lemma}[Averaging over a lattice]\label{lem:periodization}
Let $F$ be a Schwartz function feasible for the Euclidean Cohn--Elkies LP on $\R^{k+m}$ at minimal distance $r\le\min(L)$, after choosing a fiber subspace $V\simeq\R^m$ and a lattice $L\subset V$. Then
\[
  f(x,\bar y)=\sum_{\ell\in L}F(x,y+\ell)
\]
is feasible for the cylinder LP, with
\[
  \widehat{f_w}(\xi)=\covol(L)^{-1}\widehat F(\xi,w).
\]
\end{lemma}

\begin{proof}
Let $D\subset V$ be a bounded fundamental domain for $L$.  Given
$A\ge0$ and $B>m$, Schwartz decay gives a constant $C_{A,B}$ such that
\[
 |F(x,v)|\le C_{A,B}(1+|x|)^{-A}(1+|v|)^{-B}.
\]
The lattice sum of the second factor is bounded uniformly for $y\in D$.
Consequently the defining series for $f$ converges absolutely and locally
uniformly (as do the series obtained after differentiating), and
\[
 \sup_{\bar y\in V/L}|f(x,\bar y)|=O_A((1+|x|)^{-A}).
\]
In particular, $f$ is continuous and vanishes at infinity on the cylinder.
Moreover, Tonelli's theorem and the tiling $V=\bigsqcup_{\ell\in L}(D+\ell)$
give
\[
 \int_{\R^k\times D}|f(x,\bar y)|\,dx\,dy
 \le \int_{\R^{k+m}}|F(z)|\,dz<\infty,
\]
so $f$ is continuous and integrable.

If $d_X((x,\bar y),0)\ge r$, then for every representative $y$ and every
$\ell\in L$,
\[
 |(x,y+\ell)|\ge d_X((x,\bar y),0)\ge r.
\]
Every summand is therefore nonpositive, and absolute convergence permits
termwise summation to give $f(x,\bar y)\le0$.

It remains to check every Fourier channel.  For $w\in L^*$, periodicity of
$e^{-2\pi i\langle w,y\rangle}$ under translation by $L$, followed by
Tonelli--Fubini, yields
\begin{align*}
 f_w(x)
 &=\frac1{\covol(L)}\int_D\sum_{\ell\in L}F(x,y+\ell)
        e^{-2\pi i\langle w,y\rangle}\,dy\\
 &=\frac1{\covol(L)}\int_VF(x,v)
        e^{-2\pi i\langle w,v\rangle}\,dv.
\end{align*}
Taking the Fourier transform in $x$ gives, with the same convention as the
Euclidean Cohn--Elkies transform,
\[
 \widehat{f_w}(\xi)=\frac1{\covol(L)}\widehat F(\xi,w)\ge0
 \qquad(\xi\in\R^k,\ w\in L^*).
\]
Thus $f$ satisfies all space- and Fourier-side cylinder constraints.
\end{proof}

\begin{theorem}[Sharp Fibrations of Solved Lattices]\label{thm:sharp_fibrations}
Let $\Lambda\subset\R^{k+m}$ be a lattice of minimal distance $r$ with a
sharp Schwartz-class Cohn--Elkies magic function $F$. Let $L\subset\Lambda$ be a
rank-$m$ sublattice, set $V=\operatorname{span}_{\R}(L)$, and suppose that
the orthogonal projection of $\Lambda$ to $V^\perp$ is discrete. Then
$\Lambda$ is an aligned $L$-fibered packing,
the aligned $L$-fibered cylinder LP is sharp, a sharp cylinder magic
function is the fiber periodization of $F$, and the optimal base density is
\[
 \delta_B=\frac{\covol(L)}{\covol(\Lambda)}.
\]
\end{theorem}

\begin{proof}
Use the orthogonal splitting $\R^{k+m}=V^\perp\oplus V$ and let
$\pi:\R^{k+m}\to V^\perp$ be orthogonal projection.  The group
$B=\pi(\Lambda)$ is discrete by hypothesis and spans $V^\perp$, so it is a
full-rank lattice there.  Set $K=\Lambda\cap V$, the kernel of
$\pi|_\Lambda$.  The inclusion $L\subset K$ has finite index
$q=[K:L]$, and the image of $\Lambda$ in
$V^\perp\times(V/L)$ has exactly $q$ points above each point of $B$.

The lattice $K$ is primitive in $\Lambda$ because it is the kernel of the
homomorphism $\pi|_\Lambda$.  Applying the block-triangular determinant
factorization to $K\subset\Lambda$ gives
\[
 \covol(\Lambda)=\covol(K)\covol(B).
\]
Since $\covol(L)=q\covol(K)$, the base density of the induced cylinder
configuration is
\[
 \delta_B(\Lambda)=\frac{q}{\covol(B)}
 =\frac{\covol(L)}{\covol(\Lambda)}.
\]
For two distinct points of this cylinder configuration, every lift of their
difference is a vector of $\Lambda\setminus\{0\}$ and therefore has norm at
least $r$.  Hence the configuration is $r$-separated.  Also
$\min(L)\ge r$, since $L\subset\Lambda$.

The preceding observation allows us to apply Lemma~\ref{lem:periodization}
to
\[
 f(x,\bar y)=\sum_{\ell\in L}F(x,y+\ell).
\]
Thus the sum is absolutely convergent, defines a continuous, integrable
function on the cylinder, is nonpositive whenever the cylinder distance is
at least $r$, and satisfies
\[
 \widehat{f_w}(\xi)=\covol(L)^{-1}\widehat F(\xi,w)\ge0
 \quad\text{for every }(\xi,w)\in V^\perp\times L^*.
\]

We finally compute its objective value.  Sharpness of the Euclidean
certificate says
\[
  \frac{F(0)}{\widehat F(0)}=\frac1{\covol(\Lambda)}.
\]
It also forces the space-side lattice zeros.  Indeed, Poisson summation for
$\Lambda$ gives
\[
 \sum_{\lambda\in\Lambda}F(\lambda)
 =\frac1{\covol(\Lambda)}
   \sum_{\lambda^*\in\Lambda^*}\widehat F(\lambda^*).
\]
The left side is at most $F(0)$ and the right side is at least
$\widehat F(0)/\covol(\Lambda)$; sharpness makes these endpoints equal.
Since all nonzero terms on the left are nonpositive, absolute convergence
then implies $F(\lambda)=0$ for every
$\lambda\in\Lambda\setminus\{0\}$.  In particular,
\[
 f(0,0)=\sum_{\ell\in L}F(\ell)=F(0).
\]
The zero-channel identity gives
$\widehat{f_0}(0)=\covol(L)^{-1}\widehat F(0)$. The cylinder bound supplied
by $f$ is consequently
\[
 \frac{f(0,0)}{\widehat{f_0}(0)}
 =\covol(L)\frac{F(0)}{\widehat F(0)}
 =\frac{\covol(L)}{\covol(\Lambda)}=\delta_B.
\]
Since $f$ is feasible with objective
$\covol(L)/\covol(\Lambda)$, Definition~\ref{def:c-constrained-lp} and
Corollary~\ref{cor:density-bound} bound the base density of every
aligned $L$-fibered packing by this value.  The $r$-separated cylinder configuration
induced by $\Lambda$ is $B$-periodic with exactly this base density, so it
attains the upper bound, and by Lemma~\ref{lem:config-dual} its
autocorrelation is a dual certificate of the same value.  The asserted
sharpness follows.
\end{proof}

\subsection{A Torus Certificate for the Barnes--Wall Lattice}
\label{sec:bw16}

Use the normalization in which $BW_{16}$ has minimum norm $4$, so the
packing distance is $r=2$.  A standard orthogonal frame exhibits
\[
  \sqrt{2}E_8\oplus\sqrt{2}E_8\subset BW_{16}
\]
with index $16$.  The corresponding quotient is
$\Gamma\cong(\Z/2)^4$.

Recall that the optimal sphere packing bound in $\R^8$ is achieved by an
explicit Schwartz function $F_{E_8}$~\cite{Via17}, normalized by
\[
 F_{E_8}(0)=\widehat F_{E_8}(0)=1,
 \qquad
 F_{E_8}(x)\leq0\quad (|x|\geq\sqrt2),
 \qquad \widehat F_{E_8}(\xi)\geq0.
\]

\begin{theorem}[Barnes--Wall torus certificate]\label{thm:bw16-torus}
Let $L=\sqrt{2}E_8$, let $T=\R^8/L$, and consider the cylinder
$X=\R^8\times T$ with exclusion distance $2$.  Then
\[
        f(x,\bar y)=F_{E_8}(x)
\]
is a sharp cylinder magic function, and the optimal base center density is
$1$ (equivalently, the ambient center density is
$1/\covol(\sqrt2E_8)=1/16$).
The bound is attained by the Barnes--Wall fibration above.

The same function certifies the finite colored problem obtained by
restricting the phases to
\[
 L^*/L\cong E_8/2E_8,
\]
which has $256$ colors.  This problem is more general than the
$16$-colored Barnes--Wall problem, whose phase set is the glue subgroup
$\Gamma\cong(\Z/2)^4\subset E_8/2E_8$; hence the latter is sharp as well.
\end{theorem}

\begin{proof}
The covering radius of $L=\sqrt{2}E_8$ is $\sqrt{2}$.  Consequently, if
$d_X((x,\bar y),0)\geq2$, then
\[
 |x|^2\geq4-d_T(\bar y,0)^2\geq4-R(L)^2=2.
\]
The space-side sign condition therefore follows from
$F_{E_8}(x)\leq0$ for $|x|\geq\sqrt{2}$.  Because $f$ is constant in the
torus coordinate, all nontrivial fiber Fourier channels vanish, while
the trivial channel is $F_{E_8}$ and has nonnegative Fourier transform.
Thus $f$ is cylinder-feasible, and its objective value is
\[
 \frac{f(0,0)}{\widehat{f_0}(0)}
 =\frac{F_{E_8}(0)}{\widehat F_{E_8}(0)}=1.
\]

Definition~\ref{def:c-constrained-lp} and
Corollary~\ref{cor:density-bound} therefore bound the base density of every
configuration by $1$.
Projecting $BW_{16}$ to the second factor gives an $E_8$ base of center
density $1$, with the phase in the first factor determined by the glue
label.  Hence the Barnes--Wall configuration attains the bound, and by
Lemma~\ref{lem:config-dual} its autocorrelation certifies the same value
on the dual side.
Restricting the torus phase to $L^*/L$ gives the $256$-colored problem,
and restricting further to $\Gamma$ gives the $16$-colored Barnes--Wall
problem.  The same phase-independent certificate remains feasible under
both restrictions, and the Barnes--Wall configuration still attains its
value.
\end{proof}

\section{Conjectures and Relations to the Planar Sphere Packing LP}
\label{sec:collapse}

The ordinary Cohn--Elkies linear program in dimension~$2$, normalized to
minimal distance~$1$, has the hexagonal-lattice center density
$2/\sqrt3$ as its natural target.  It is conjectured that the linear
program is sharp at this value, although no sharp planar magic function is
known.  We present two fibered problems that are controlled by precisely
this unresolved scalar question.

\begin{conjecture}[Sharpness of the planar LP]\label{conj:planar-lp}
The two-dimensional Cohn--Elkies linear program at minimal distance~$1$
has value $2/\sqrt3$, attained by a magic function for the hexagonal
lattice.
\end{conjecture}

\subsection{The $E_6$ over $D_4$ collapse theorem}
\label{sec:e6d4_collapse}

Consider the $(\Z/2)^2$-colored packing problem on $\R^2$ corresponding to
the $E_6$ over $D_4$ fibration.  The same-color minimal distance is
$\sqrt2$, all three cross-color distances are $1$, and the conjectured
optimal base is the hexagonal lattice with center density $2/\sqrt3$.
We will show that the linear program for this colored packing problem is sharp
if and only if the two-dimensional Cohn--Elkies linear program is sharp for
the hexagonal lattice.

Throughout this subsection, $B$ is the hexagonal lattice of minimal
distance $1$, identified with the $A_2$ root lattice rescaled to this
normalization.  Let $\alpha_1,\alpha_2$ be its unit simple roots, so that
\[
 |\alpha_1|=|\alpha_2|=1,\qquad
 \langle\alpha_1,\alpha_2\rangle=-\frac12.
\]
For integers $m,n$, set
\[
 Q(m,n):=m^2-mn+n^2.
\]
Then
\[
 |m\alpha_1+n\alpha_2|^2=Q(m,n).
\]
The minimal vectors are $\pm\alpha_1$, $\pm\alpha_2$, and
$\pm(\alpha_1+\alpha_2)$.
A positive integer is called \emph{L\"oschian} if it is represented by
$Q$, and the nonzero distance set of $B$ is
$\{\sqrt{Q(m,n)}:(m,n)\in\Z^2\setminus\{(0,0)\}\}$, equivalently
$\{\sqrt N:N\ \hbox{is L\"oschian}\}$.  The color group is
$C=(\Z/2)^2$.
We identify $\widehat C$ with $C$ via
\[
 v\longmapsto\chi_v,\qquad \chi_v(y)=(-1)^{v\cdot y},
\]
and write $\chi=v$ when using this identification.

A \emph{colored magic kernel} is a primal function
$f:\R^2\times C\to\R$ that certifies the value $2/\sqrt3$.  Its character
channels are
\[
 f_\chi(x)=\frac1{|C|}\sum_{y\in C}
 f(x,y)\overline{\chi(y)}.
\]
Thus feasibility means $\widehat{f_\chi}\geq0$ on $\R^2$ and
$f(x,y)\leq0$ for $|x|\geq d_y$, and the certified value is
$f(0,0)/\widehat{f_{\mathbf1}}(0)$
(Definition~\ref{def:c-constrained-lp}).

Since $1<\sqrt2\leq\sqrt3$, a coloring $c:B\to C$ is \emph{valid} for
this program if and only if adjacent points of the triangular graph (the
nearest-neighbor graph of $B$) receive different colors, and every valid
coloring is then an optimal configuration, of density $2/\sqrt3$.  The
proof of the collapse theorem rests on the flexibility of such
colorings.

\begin{lemma}[Flexibility of the $(\Z/2)^2$ colorings]\label{lem:flexcolor}
The following colorings of $B$ are valid, for arbitrary
$\varphi,\psi:\Z\to\Z/2$ and any injection
$\iota:\Z/3\hookrightarrow(\Z/2)^2$ with unused color $D$:
\begin{enumerate}
\item[(i)] $c(m\alpha_1+n\alpha_2)
=\big(m+\varphi(n),\,n\big)\bmod2$, and symmetrically
$\big(m,\,n+\psi(m)\big)\bmod2$;
\item[(ii)] the embedded $3$-colorings
$c=\iota\big((m+n)\bmod3\big)$;
\item[(iii)] the modification of (ii) that recolors every point of a set
$S\subset B$ containing no adjacent pair to the color $D$.
\end{enumerate}
\end{lemma}

\begin{proof}
(i) The six minimal differences change $(m,n)$ by $(\pm1,0)$,
$(0,\pm1)$, $\pm(1,1)$: the first flips the first displayed coordinate
($\varphi(n)$ is unchanged), the other four flip the second.  (ii) The
linear form $m+n$ changes by $1$, $1$, $2\bmod3$ along the minimal
differences, never by $0$.  (iii) A distance-$1$ pair contains at most
one point of $S$; if none, properness is (ii); if one, its colors are
$D$ and a value of $\iota$, which differ.  Same-colored pairs are
$D$--$D$ pairs (within $S$, hence at distance $\geq\sqrt3$) or
$\iota$--$\iota$ pairs (same class modulo the index-$3$ sublattice
$\ker(m+n\bmod3)$, hence at distance
$\in\sqrt{3\cdot\hbox{L\"oschian}}\geq\sqrt3$), and $\sqrt3>\sqrt2$.
\end{proof}

\begin{lemma}[Rotationally averaged complementary slackness]
\label{lem:e6d4-averaged-slackness}
Let $c:B\to C$ be any valid coloring and put
\[
 \omega_c=\sum_{p\in B}\delta_{(p,c(p))}.
\]
Every subsequential volume-averaged autocorrelation
\[
 \gamma_c=\lim_{j\to\infty}\frac1{R_j^2}
 \big(\omega_c|_{Q_{R_j}\times C}\big)
 *\widetilde{\big(\omega_c|_{Q_{R_j}\times C}\big)}
\]
where $\widetilde\mu(E)=\mu(-E)$, gives, after averaging over orthogonal
transformations,
\[
 \overline\gamma_c
 =\frac{\sqrt3}{2}\int_{O(2)}(R\times\mathrm{id}_C)_*\gamma_c\,dR,
\]
a dual-feasible measure of value $2/\sqrt3$, where $dR$ is normalized
Haar measure.  If $f$ is a magic kernel, then $\overline\gamma_c$ is dual
optimal and $f$ satisfies
\[
 f=0\quad\hbox{on }\operatorname{supp}\overline\gamma_c\setminus\{0\},
 \qquad
 \widehat f=0\quad\hbox{on }
 \operatorname{supp}\widehat{\overline\gamma_c}
 \setminus\{(0,\mathbf1)\}.
\]

In particular, if $c$ is $M$-periodic, then for every $R\in O(2)$,
\begin{align}
 f\big(R(p-q),c(p)-c(q)\big)&=0
 &&(p\ne q\in B), \label{eq:e6d4-space-slack}\\
 \widehat{f_\chi}(R\xi)&=0
 &&\big(\xi\in M^*,\ \chi\in\widehat C,\
       (\xi,\chi)\ne(0,\mathbf1),\ \sigma_\chi(\xi)\ne0\big),
 \label{eq:e6d4-fourier-slack}
\end{align}
where
\[
 \sigma_\chi(\xi)
 =\sum_{p\in B/M}\chi(c(p))\,\e{\langle\xi,p\rangle}.
\]
\end{lemma}

\begin{proof}
Subsequential autocorrelations exist because $\omega_c$ is translation
bounded.  They are positive-definite measures supported on the identity
and the admissible difference set.  Projecting to the trivial character
forgets the colors and gives the autocorrelation of $B$, so the atom at
the identity has mass $2/\sqrt3$ and the trivial Fourier atom has mass
$(2/\sqrt3)^2$.  Thus $(\sqrt3/2)\gamma_c$ is dual feasible with
objective value $2/\sqrt3$.

The colored program is invariant under $O(2)$, so averaging preserves dual
feasibility and the objective value.  If a magic kernel is given, then it
has the same value, and hence the averaged measure is dual optimal.  The
displayed support conditions follow from complementary slackness, by the
same mollification argument as in Remark~\ref{rem:config-slackness}.

For periodic $c$, every difference of two points gives an atom of
$\gamma_c$, while Lemma~\ref{lem:config-dual} gives a Fourier atom at
$(\xi,\chi)$ precisely when $\sigma_\chi(\xi)\ne0$.  Averaging carries
each such atom through its full $O(2)$-orbit, which proves
\eqref{eq:e6d4-space-slack} and \eqref{eq:e6d4-fourier-slack}.
\end{proof}

\begin{theorem}[Collapse of the $E_6$ over $D_4$ Two-Point Program]
\label{thm:e6d4collapse}
For the $(\Z/2)^2$-colored program:
\begin{enumerate}
  \item Every magic kernel has $f_\chi\equiv0$ for all nontrivial
  $\chi\ne\mathbf1$.  Equivalently,
  $f(x,y)=h(x):=f_{\mathbf1}(x)$ for all $y\in C$.

  \item A magic kernel exists if and only if the planar Cohn--Elkies LP at
  minimal distance~$1$ admits a magic function sharp for the hexagonal
  lattice, as in Conjecture~\ref{conj:planar-lp}.
\end{enumerate}
\end{theorem}

\begin{proof}
Given a planar magic function $h$, put $f(x,y)=h(x)$ for every
$y\in C$, or equivalently $f_{\mathbf1}=h$ and $f_\chi=0$ otherwise.
Then $f$ is feasible for the colored $E_6/D_4$ linear program and gives
the same bound, hence is sharp.

Conversely, assume $f$ is a magic kernel.  We will show that nontrivial channels $f_{\chi}$ vanish by showing that they must both be band-limited and have radial shells of zeros according to square roots of L\"oschian $N\geq3$.

Using the complementary slackness lemma, we first note that a family of colorings imposes
conditions on the Fourier spectrum and, in the limit, forces $f_\chi$ to
be bandlimited for every nontrivial $\chi$: in fact,
$\widehat{f_\chi}(\xi)=0$ whenever $|\xi|\geq1/2$.  To achieve this, we use
the colorings in Lemma~\ref{lem:flexcolor}(i) with $\varphi$ the indicator
of $P\Z$, for $P\geq4$ even.  For this choice, the period lattice is
$M=2\alpha_1\Z\oplus P\alpha_2\Z$, so $\xi\in M^*$ is parametrized by
$(t_1,t_2):=(\langle\xi,\alpha_1\rangle,\langle\xi,\alpha_2\rangle)\in
\tfrac12\Z\times\tfrac1P\Z$.  For $\chi=(1,0)$ the amplitude factorizes as
\[
 \sigma_\chi(\xi)=\Big[\sum_{m\bmod2}(-1)^m\e{mt_1}\Big]
 \Big[\sum_{n\bmod P}(-1)^{\varphi(n)}\e{nt_2}\Big]
 =\big(1-\e{t_1}\big)\big(P\,\mathbf1_{t_2\in\Z}-2\big),
\]
which is nonzero exactly when $t_1\in\tfrac12+\Z$ (the second factor
never vanishes for $P\geq3$).  For $\chi=(1,1)$ replace $\varphi$ by
$\varphi(n)+n$, which is again an arbitrary function; for $\chi=(0,1)$ use
the $\psi$-family with the roles of $\alpha_1,\alpha_2$ exchanged.  In
every case there is a minimal vector $e\in\{\alpha_1,\alpha_2\}$ such
that \eqref{eq:e6d4-fourier-slack} makes $\widehat{f_\chi}$ vanish on the
$O(2)$-orbits of all points of the lines
$\{\xi:\langle\xi,e\rangle\in\tfrac12+\Z\}$ with
$\langle\xi,e'\rangle\in\tfrac1P\Z$ ($e'$ the complementary basis
vector).  As $P\to\infty$, these points are dense on the line
$\{\xi:\langle\xi,e\rangle=1/2\}$, whose $O(2)$-orbit is exactly
$\{\xi:|\xi|\geq1/2\}$.  Continuity therefore gives
$\operatorname{supp}\widehat{f_\chi}\subseteq\{|\xi|\leq1/2\}$.

The next step is to find a different choice of colorings to ensure that $f_{\chi}$ must have superlinearly many zeros.
We claim that for every $y\in C$, every L\"oschian $N\geq3$, and every $x_0$ with $|x_0|^2=N$, all four
color-difference functions vanish: $f(x_0,y)=0$.  It suffices to
realize, for each $y$ and each $d\in B$ with
$|d|^2=N$, a valid periodic coloring with a pair of color difference
$y$ at difference $d$; then \eqref{eq:e6d4-space-slack} gives the zero
on the full circle of radius $|d|$.  Write
$j_0=(d_1+d_2)\bmod3$ for the $3$-coloring difference along
$d=d_1\alpha_1+d_2\alpha_2$.
\begin{itemize}
\item $y\ne0$, $j_0\ne0$: use Lemma~\ref{lem:flexcolor}(ii) with
$\iota$ chosen so that $\iota(j+j_0)-\iota(j)=y$ for some $j$.
\item $y\ne0$, $j_0=0$: use (iii) with $S=x_0+d+\Lambda$ for a
sublattice $\Lambda\subset B$ of minimal distance $\geq|d|+2$ (so $S$
has no adjacent pair), and $\iota$ chosen with
$\iota\big((x_0)_1+(x_0)_2\bmod3\big)=D+y$: the pair $(x_0,x_0+d)$
has colors $D+y$ and $D$.
\item $y=0$: use (iii) with
$S=(x_0+\Lambda)\cup(x_0+d+\Lambda)$, again with
$\min(\Lambda)\geq|d|+2$; the pair $(x_0,x_0+d)$ is $D$--$D$ at
distance $\sqrt N\geq\sqrt3$.
\end{itemize}
Inverting the character sum,
$f_\chi(x_0)=|C|^{-1}\sum_y
\overline{\chi(y)}\,f(x_0,y)
=0$ for every $\chi$.
The number of L\"oschian numbers up to $X$ grows like
$cX/\sqrt{\log X}$ (Landau--Bernays~\cite{Ber12}; for the weaker estimate
needed here, the identity $u^2+3v^2=Q(u+v,2v)$ shows that the L\"oschian
numbers contain the values of $u^2+3v^2$, of which there are
$\gg X^{1-\varepsilon}$ up to $X$).  Consequently, for every unit vector
$\hat x$ and every $\chi$, the restriction
$t\mapsto f_\chi(t\hat x)$ has $\gg R^{2-\varepsilon}$ zeros in
$|t|\leq R$.

We now know that $f_\chi$ is band-limited for nontrivial $\chi$, so
any unit vector direction $t \mapsto f_\chi(t\hat x)$, extends to an
entire function of finite exponential type.
On the other hand, a nonzero entire function of finite exponential type
has only $O(R)$ zeros in $|t|\leq R$, counted with multiplicity, by Jensen's formula.
But by the previous paragraph, we know that each such restriction has superlinearly many zeros.
This contradiction shows that $f_\chi$ must vanish for every nontrivial $\chi$.

With $h:=f_{\mathbf1}$ we now have $f(x,y)=h(x)$ for all $y$; feasibility for any $y\ne0$ gives
$h\leq0$ beyond $1$, $\widehat h=\widehat{f_{\mathbf1}}\geq0$, and the
value is $h(0)/\widehat h(0)=2/\sqrt3$: a planar magic function at
distance $1$, sharp for the hexagonal lattice.
\end{proof}

\begin{corollary}[Cylinder collapse]\label{cor:e6d4cyl}
The cylinder LP for aligned $D_4$-fibered packings of $\R^6$
($r=\sqrt2$) admits a sharp magic function if and only if
Conjecture~\ref{conj:planar-lp} holds.
\end{corollary}

\begin{proof}
The cylinder LP gives a bound greater than or equal to that of the colored packing LP by functoriality, and any magic function for the hexagonal lattice produces a magic function $f(x,y) = h(x)$ as before, so the result follows from Theorem~\ref{thm:e6d4collapse} and Theorem~\ref{thm:lca-primal-attainment}.
\end{proof}

\subsection{The rigid $D_4$ over $A_2$ problem}
\label{sec:d4a2_rigidity}

The collapse theorem relies on the flexibility of $(\Z/2)^2$-colorings of
the hexagonal lattice.  The $D_4$ over $A_2$ problem behaves differently:
its triangular base is rigidly three-colorable.

\paragraph{Observation.}
Every proper $3$-coloring of the triangular graph is determined by its
values on any one elementary triangle.  Indeed, such a triangle uses
all three colors.  Each adjacent triangle shares an edge with it, so
the color of its third vertex is forced; the conclusion follows by
propagating across the connected graph of elementary triangles.

Here $C=\Z/3$, the same-color distance is $\sqrt2$, the
cross-color distance is $2/\sqrt3$, and the target center density is
$\sqrt3/2$.  A valid coloring is again exactly a proper $3$-coloring of
the triangular graph, so by the observation above the optimal
configurations form a single orbit under isometries and color
permutations.  Because the coloring cannot be varied to force further
zeros, the collapse proof does not apply.

\begin{conjecture}[Sharpness of $D_4$ over $A_2$]\label{conj:d4a2}
The two-point $\Z/3$-colored LP for $D_4$ over $A_2$ and the cylinder
LP for aligned $A_2$-fibered packings of $\R^4$ at exclusion distance
$\sqrt2$ both have value $\sqrt3/2$, attained by a magic kernel and a
cylinder magic function, respectively.
\end{conjecture}

\begin{proposition}[Planar sharpness implies $D_4/A_2$ sharpness]
\label{prop:d4a2-planar}
Conjecture~\ref{conj:planar-lp} implies
Conjecture~\ref{conj:d4a2}.
\end{proposition}

\begin{proof}
Assume that $h$ is a planar magic function at minimal distance~$1$,
sharp for the hexagonal lattice.  Put $a=2/\sqrt3$ and
$h_a(x)=h(x/a)$.  Then $h_a\leq0$ for $|x|\geq a$ and
\[
 \frac{h_a(0)}{\widehat h_a(0)}
 =\frac{1}{a^2}\frac{h(0)}{\widehat h(0)}
 =\frac{3}{4}\frac{2}{\sqrt3}
 =\frac{\sqrt3}{2}.
\]
For the colored problem, take the trivial character channel to be
$h_a$ and every nontrivial channel to vanish.  All three
color-difference kernels then equal $h_a$.  This is feasible because
the cross-color threshold is $a=2/\sqrt3$, while the same-color
threshold $\sqrt2$ is larger.  The properly $3$-colored hexagonal
lattice attains the resulting value $\sqrt3/2$.

For the cylinder problem, normalize the $A_2$ fiber to have minimal
distance $\sqrt2$; its covering radius is $\sqrt{2/3}$.  The
fiber-constant function
\[
 f(x,\bar y)=h_a(x)
\]
has only the nonnegative trivial fiber Fourier channel.  Moreover,
if $d_X((x,\bar y),0)\geq\sqrt2$, then
\[
 |x|^2\geq 2-d_T(\bar y,0)^2
 \geq 2-\frac23=\frac43=a^2,
\]
so the required space-side sign condition follows from that of
$h_a$.  Thus $f$ certifies the value $\sqrt3/2$, which is attained by
the $D_4$ fibration.
\end{proof}

\begin{remark}[Numerical footprint]\label{rem:e6d4numerics}
Theorem~\ref{thm:e6d4collapse} was found through, and retroactively
explains, numerics: at Laguerre-polynomial degree $20$ for the $D_4$-fibered linear program in $6$ dimensions,
the nontrivial channels of the optimizer are very close to $0$, with amplitude
$\sup_\xi|\widehat{f_\chi}|\leq5\times10^{-7}$.  By contrast, the finite-degree polynomial
optimizers for the $D_4$ over $A_2$ problem have nontrivial
cross-channels.
\end{remark}

\section{Discrete Reduction Duals}\label{sec:reduction}

For a rigorous lower bound on the translation-invariant colored LP, we
extend Li's domain-restriction argument~\cite{Li25} while retaining the
finite color group.

Let $C$ be a finite abelian color group, with exclusion distance $d_\delta$
for color difference $\delta\in C$.  Choose a scale $s>0$ and an integer
$m$ satisfying the no-wrap condition
\[
 m\ge 2s\max_\delta d_\delta,
\]
and set
\[
 H=(\Z/m\Z)^k\times C.
\]
We use centered representatives of $(\Z/m\Z)^k$ and the unnormalized finite
Fourier transform on $H$.

\begin{theorem}[Colored discrete reduction]\label{thm:colored-discrete}
Suppose $\mu:H\to\R$ satisfies
\begin{enumerate}
 \item $\mu(0,0)=1$;
 \item $\mu(x,\delta)=0$ when $(x,\delta)\ne(0,0)$ and
       $|x|<s d_\delta$;
 \item $\mu(x,\delta)\ge0$ when $|x|\ge s d_\delta$; and
 \item $\widehat\mu(\eta,\chi)\ge0$ for every character
       $(\eta,\chi)\in\widehat H$.
\end{enumerate}
Then the translation-invariant colored LP in the original base coordinates
is at least
\[
  \frac{s^k}{m^k}\widehat\mu(0,\mathbf1)
  =\frac{s^k}{m^k}\sum_{h\in H}\mu(h).
\]
\end{theorem}

\begin{proof}
Write $P(d)$ for the colored LP value with distance data
$(d_\delta)_{\delta\in C}$.
First dilate the base coordinates by $s$, so that the exclusion distances
are $s d_\delta$. Since $P$ is a $k$-dimensional center density, this
convention gives
\[
 P(sd)=s^{-k}P(d),\qquad\text{or equivalently}\qquad
 P(d)=s^kP(sd).
\]
Let
\[
 \T_m^k=\R^k/m\Z^k
\]
with normalized Haar measure, and view
$H=(\Z/m\Z)^k\times C$ as a finite subgroup of $\T_m^k\times C$ using
centered integer representatives.  The no-wrap hypothesis implies that an
admissible difference in $H$ maps to an admissible difference in the torus.

Set
\[
 \alpha_H=\sum_{h\in H}\mu(h)\delta_h,
 \qquad z=\widehat\mu(0,\mathbf1)=\sum_{h\in H}\mu(h).
\]
The first three assumptions give
$\alpha_H=\delta_0+\nu_H$, where $\nu_H$ is nonnegative and supported on
the admissible nonzero differences.  The fourth assumption says that
$\alpha_H$ is positive definite.  More precisely,
\[
 \delta_0+\nu_H-zm_H\in K^*(H),
\]
because its Fourier transform agrees with $\widehat\mu$ away from the
trivial character and vanishes at the trivial character.  Thus $\alpha_H$
is dual feasible for the finite-group program with objective value $z$.

Apply Proposition~\ref{prop:lca-functoriality} to the inclusion
\[
 H\lhook\joinrel\longrightarrow \T_m^k\times C.
\]
Its proof shows directly that the pushforward of $\alpha_H$ gives a torus
dual measure with objective at least $z$.  On the dual side, the adjoint of
averaging over the lattice $m\Z^k$ is periodic lifting.  Periodically lift
the torus certificate along the quotient map
\[
 \R^k\times C\longrightarrow\T_m^k\times C.
\]
The lift of a positive-definite distribution on the torus is positive
definite, as is immediate from its nonnegative Fourier series, while
normalized Haar measure lifts to $m^{-k}m_{\R^k\times C}$.  Hence the lifted
dual measure $\alpha_{\R}:=\delta_0+\nu_{\R}$ satisfies
\[
 \alpha_{\R}-\frac{z}{m^k}m_{\R^k\times C}
 \in K^*(\R^k\times C).
\]
Writing a tilde for periodic lift, the nonnegative measure here is
\[
 \nu_{\R}
 =\widetilde{\iota_*\nu_H}
  +\sum_{n\in m\Z^k\setminus\{0\}}\delta_{(n,0)},
\]
where $\iota:H\hookrightarrow\T_m^k\times C$ is the inclusion.  Every
point in the support of $\nu_{\R}$ lies in
the admissible-difference set: centered representatives minimize the norm
of their lifts, and the no-wrap hypothesis handles the nonzero lattice
points.  Thus $\alpha_{\R}$ is dual feasible with objective at least
$z/m^k$, and weak duality gives $P(sd)\ge z/m^k$.

Finally, dilation of the Euclidean base gives $P(d)=s^kP(sd)$.  Therefore
\[
 P(d)\ge \frac{s^k}{m^k}z,
\]
as claimed.
\end{proof}

\subsection{A certified gap for $D_5/A_3$}

For $D_5$ over $A_3$, use the root-lattice normalization
$r^2=2$. The four cosets of $A_3^*/A_3\cong\Z/4\Z$ have squared depths
$\rho_\delta^2=0,3/4,1,3/4$. Hence the base exclusion distances,
determined by $d_\delta^2+\rho_\delta^2=r^2$, have squares
\[
 (2,5/4,1,5/4).
\]

\begin{theorem}[Certified $D_5/A_3$ gap]\label{thm:d5a3-gap}
The translation-invariant two-point colored LP for $D_5/A_3$ has value at
least
\[
 \frac{289801}{288000}
 =1+\frac{1801}{288000}>1.
\]
Consequently, this LP is not sharp for $D_5/A_3$ and cannot prove the
center-density conjecture of Cohn and Rajagopal.
\end{theorem}

\begin{proof}
The exact rational certificate supplied in the accompanying
reproducibility package satisfies the hypotheses of
Theorem~\ref{thm:colored-discrete} and has the displayed objective value.
The package verifies all certificate conditions using exact arithmetic.
\end{proof}

\section{Duality for the Program on $\R^k\times A$}\label{sec:duality}

This section proves the duality theorems stated in
Section~\ref{sec:formulations}.

The strong duality theorem of Berdysheva, Farkas, Ga\'al, Ramabulana, and
R\'ev\'esz \cite[Theorem~5.3 and Remark~5.4]{BFGRR26} for the Delsarte
problem on locally compact abelian groups specializes to
Theorem~\ref{thm:strong-duality-lca}.  Section~\ref{app:classes} reconciles the
relevant spaces of functions and measures used in their work and ours.
That reconciliation rests on weak duality for $K(X)$, for which
Proposition~\ref{prop:app-mollify} adapts the mollification argument from~\cite{CdLS22}.
Dual attainment, Theorem~\ref{thm:dual-attainability}, is not part of
\cite{BFGRR26}, which equates the two optimal values but makes no
assertion that either is attained.  It generalizes
\cite[Proposition~3.6]{CdLS22}, the case $X=\R^n$, and follows in
Section~\ref{app:attainability} from vague compactness of the dual
feasible set and the continuity of the Fourier transform on the cone of
positive definite measures.  The primal attainment theorem,
Theorem~\ref{thm:lca-primal-attainment}, is proved in
Section~\ref{app:primal-attainment} by adapting
\cite[Proposition~A.1]{CdLS22} as well.

\subsection{Conventions}

Feasible functions are real-valued with $\widehat f\geq0$, hence have
real and even Fourier transforms and are themselves even; we work with
even real-valued functions throughout.  Apart from the mollification
in the proof of Proposition~\ref{prop:app-mollify}, every construction
below uses only convolutions of indicator functions on $A$ and smooth
compactly supported functions on the base $\R^k$, so $A$ may be an
arbitrary compact abelian group.  In the mollification we write
$\mathcal S(X)$ for the even real smooth functions on $X$ all of whose
derivatives decay faster than every power of $(1+|x|)$, and
$\mathcal S(\widehat X)$ for its Fourier image; tempered distributions
are the continuous duals.  This is stated for $A$ a compact abelian Lie
group, that is, a product of a torus and a finite abelian group, which
covers every application in this paper.  In general one first averages
the fiber variable over a closed subgroup $H$ with $A/H$ Lie, which is
possible inside any neighborhood of the identity, and applies the
argument on $\R^k\times(A/H)$; every Bruhat--Schwartz function on $X$
factors through such a quotient~\cite{Osb75}.

\begin{lemma}[Feasibility]\label{lem:app-feasible}
If $0\notin C$, there is a continuous integrable $f$ on $X$, smooth and
rapidly decreasing in the base variable, with $f\leq0$ on $C$,
$\widehat f\geq0$, and $\widehat f(0,\mathbf1)=1$.
\end{lemma}

\begin{proof}
Since $C$ is closed and avoids the identity, there are
$\varepsilon>0$ and an open neighborhood $V\ni e$ in $A$ with
$C\cap(B^k_\varepsilon\times V)=\emptyset$.  Choose an even Schwartz
function $f_1$ on $\R^k$ with $f_1\leq0$ outside $B^k_\varepsilon$,
$\widehat{f_1}\geq0$, and $\widehat{f_1}(0)>0$ (a rescaled feasible
function for the Euclidean sphere-packing bound), and choose an open
$V_0\ni e$ with $V_0V_0^{-1}\subseteq V$.  Put
$\varphi=\mathbf 1_{V_0}/m_A(V_0)$ and
$\psi=\varphi*\widetilde\varphi$, where
$\widetilde\varphi(a)=\varphi(a^{-1})$: then $\psi$ is continuous,
$\psi\geq0$ is supported
in $V$, $\widehat\psi=|\widehat\varphi|^2\geq0$, and
$\widehat\psi(\mathbf1)=\big(\int_A\varphi\big)^2=1>0$.  The product
$f(x,a)=f_1(x)\psi(a)$ is continuous and integrable with
$\widehat f=\widehat{f_1}\otimes\widehat\psi\geq0$ and
$\widehat f(0,\mathbf1)>0$.  At every point of $C$ either
$|x|\geq\varepsilon$, where $f_1\leq0$ and $\psi\geq0$, or
$a\notin V$, where $\psi=0$; in both cases $f\leq0$.  Rescale so that
$\widehat f(0,\mathbf1)=1$.
\end{proof}

\subsection{Comparison of classes}\label{app:classes}

In the normalization used in the Delsarte literature the primal reads
\[
 \Delta(C)=\sup\left\{\int_Xf\,dm_X:
 f\in K(X),\ f(0)=1,\ f\leq0\hbox{ on }C\right\},
\]
and we write $\Delta^{\infty,1}(C)$ for the same supremum restricted to
the Wiener amalgam space
$C^{\infty,1}(X)=\{f\in C(X):\sum_{\ell\in\Z^k}\sup_{B+\ell}|f|<\infty\}$
with $B=[0,1)^k\times A$, in which \cite{BFGRR26} sets up the primal
problem.  Since $C^{\infty,1}(X)\subseteq C(X)\cap L^1(X)$
\cite[Section~3]{BFGRR26}, a positive definite element of
$C^{\infty,1}(X)$ lies in $K(X)$, so
$\Delta^{\infty,1}(C)\leq\Delta(C)$; and
$\Delta^{\infty,1}(C)>0$ because the feasible function of
Lemma~\ref{lem:app-feasible} lies in $C^{\infty,1}(X)$.  The two
normalizations are exchanged by $f\mapsto f/f(0)$ and
$f\mapsto f/\int_Xf$: a feasible $f$ for
\eqref{eq:primal_lca_val} is nonzero, so $f(0)>0$ and $f/f(0)$ is
admissible for $\Delta(C)$ with integral $1/f(0)$; conversely
$f/\int_Xf$ is feasible for \eqref{eq:primal_lca_val} whenever
$\int_Xf>0$, and admissible functions with $\int_Xf=0$ do not affect a
positive supremum.  Hence
\begin{equation}\label{eq:app-reciprocity}
 P_X(C)=1/\Delta(C),
 \qquad\hbox{with }1/\infty=0.
\end{equation}

\begin{proof}[Proof of Theorem~\ref{thm:strong-duality-lca}]
Apply \cite[Theorem~5.3]{BFGRR26} to the locally compact abelian group
$X=\R^k\times A$ with $\Omega=X\setminus C$, taking $\rho=-m_X$ and
$\sigma=\delta_0$ as in \cite[Remark~5.4]{BFGRR26}.  The hypotheses
hold: $\Omega$ is open and contains $0$ because $C$ is closed and
$0\notin C$; both $m_X$ and $\delta_0$ are translation bounded and
even; and $\delta_0$ is strictly positive definite in the sense of
\cite{BFGRR26}, its Fourier transform being the constant function $1$,
which has no zero on $\widehat X$.  The conclusion is
\[
 \Delta^{\infty,1}(C)
 =\inf\{s\in\R:\ s\delta_0-m_X\in\mathcal M(C)\},
\]
where $\mathcal M(C)$ is the set of translation bounded real Radon
measures of the form $\nu-\kappa$ with $\kappa\geq0$ supported in $C$
and $\nu$ positive definite.

Each admissible $s>0$ produces a measure feasible for
\eqref{eq:dual_lca_val}: writing $s\delta_0-m_X=\nu-\kappa$ and
dividing by $s$,
\[
 \mu:=\delta_0+\tfrac1s\kappa=\tfrac1s(\nu+m_X)
\]
is nonnegative, supported in $\{0\}\cup C$, and has $\mu(\{0\})=1$
because $0\notin C$.  A translation bounded positive definite measure
has nonnegative Fourier transform, and
$\widehat{m_X}=\delta_{(0,\mathbf1)}$, so
$\widehat\mu=\frac1s(\widehat\nu+\delta_{(0,\mathbf1)})$ is a
nonnegative tempered measure whose atom at $(0,\mathbf1)$ has mass at
least $1/s$.  Since $\Delta^{\infty,1}(C)>0$, admissible values of $s$
decrease to $\Delta^{\infty,1}(C)$ through positive numbers, whence
$D_X(C)\geq1/\Delta^{\infty,1}(C)$; when $\Delta^{\infty,1}(C)=\infty$
this says only $D_X(C)\geq0$, which also holds because $\mu=\delta_0$
is dual feasible with objective value $0$.  Combining this with
\eqref{eq:app-reciprocity} and with the weak duality of
Proposition~\ref{prop:app-mollify},
\[
 D_X(C)\ \geq\ \frac1{\Delta^{\infty,1}(C)}
 \ \geq\ \frac1{\Delta(C)}
 \ =\ P_X(C)\ \geq\ D_X(C),
\]
so all the inequalities are equalities and $P_X(C)=D_X(C)$.
\end{proof}

\subsection{Weak duality for $K(X)$}

\begin{proposition}\label{prop:app-mollify}
Suppose $T=\delta_0+\mu$ with $\mu\geq0$ a tempered measure on $X$ and
$\widehat T=z\delta_{(0,\mathbf1)}+\nu$ with $\nu\geq0$ on
$\widehat X$.  If $f:X\to\R$ is continuous and integrable with
$f\leq0$ on $\operatorname{supp}\mu$, $\widehat f\geq0$, and
$\int_Xf\,dm_X=1$.  Then $f(0)\geq z$; in particular
$D_X(C)\leq P_X(C)$.  If $f(0)=z$, then $f$ is $\mu$-integrable and
$\widehat f$ is $\nu$-integrable, with
\[
 \int_X f\,d\mu=0
 \qquad\hbox{and}\qquad
 \int_{\widehat X}\widehat f\,d\nu=0,
\]
so that $f=0$ $\mu$-almost everywhere and $\widehat f=0$
$\nu$-almost everywhere.  If moreover $\mu$ and $\nu$ are atomic, then
$f=0$ on $\operatorname{supp}\mu$ and $\widehat f=0$ on
$\operatorname{supp}\nu$.
\end{proposition}

\begin{proof}
This is \cite[Proposition~3.5]{CdLS22} with a mollifier adapted to the
compact factor.  Let $\varphi:\R^k\to\R$ be smooth, even, and
nonnegative with $\operatorname{supp}\varphi\subseteq B^k_1$,
$\widehat\varphi\geq0$, and $\widehat\varphi(0)=1$, and put
$\varphi_\varepsilon(x)=\varepsilon^{-k}\varphi(x/\varepsilon)$.  On
the fiber, let $\kappa_i=\psi_i*\widetilde{\psi_i}$ with
$\psi_i\in C^\infty(A)$ nonnegative and supported in a neighborhood
basis of the identity, normalized by
$\int_A\kappa_i\,dm_A=1$; then $\widehat{\kappa_i}=|\widehat{\psi_i}|^2\geq0$,
$\widehat{\kappa_i}(\mathbf1)=1$, and $u*\kappa_i\to u$ uniformly for
every continuous $u$ on $A$.  Writing
$\Phi_{\varepsilon_1}(x,a)=\widehat\varphi(\varepsilon_1x)$, positive
definiteness bounds $|f|$ by $f(0)$, so $f\cdot\Phi_{\varepsilon_1}$ is
rapidly decreasing and
\[
 h:=\big(f\cdot\Phi_{\varepsilon_1}\big)*
 \big(\varphi_{\varepsilon_2}\otimes\kappa_i\big)
\]
is Schwartz with
\[
 \widehat h(\xi,\chi)
 =\big(\widehat f(\cdot,\chi)*\varphi_{\varepsilon_1}\big)(\xi)\,
 \widehat\varphi(\varepsilon_2\xi)\,\widehat{\kappa_i}(\chi)\geq0,
\]
since multiplication by $\Phi_{\varepsilon_1}$ convolves each character
channel of $\widehat f$ with $\varphi_{\varepsilon_1}$ in $\xi$ alone.
As $\mu$ is a nonnegative tempered measure, $(1+|x|)^{-K}$ is
$\mu$-integrable for some $K$ \cite{Sch66} ($A$ being compact, the
Euclidean statement applies verbatim), so $f\cdot\Phi_{\varepsilon_1}$
is $\mu$-integrable, and pairing $T$ with $h$ gives
\[
 h(0)+\int h\,d\mu
 =z\,\widehat h(0,\mathbf1)+\int\widehat h\,d\nu
 \geq z\,\big(\widehat f(\cdot,\mathbf1)*\varphi_{\varepsilon_1}\big)(0),
\]
the right side being independent of $\varepsilon_2$ and $i$.  Since $h$
is a weighted average of $f\cdot\Phi_{\varepsilon_1}$ over shrinking
neighborhoods, letting $\varepsilon_2\to0$ and $i\to\infty$ yields, by
dominated convergence exactly as in \cite{CdLS22},
\[
 f(0)+\int f\cdot\Phi_{\varepsilon_1}\,d\mu
 \geq z\,\big(\widehat f(\cdot,\mathbf1)*\varphi_{\varepsilon_1}\big)(0).
\]
The left integral is nonpositive because $f\leq0$ on
$\operatorname{supp}\mu$ and $\Phi_{\varepsilon_1}\geq0$, while as
$\varepsilon_1\to0$ the right side tends to
$z\,\widehat f(0,\mathbf1)=z$, the trivial channel of $\widehat f$
being continuous.  Hence $f(0)\geq z$, and taking
$\alpha=\delta_0+\mu$ dual feasible with
$z=\widehat\alpha(\{(0,\mathbf1)\})$ gives $P_X(C)\geq D_X(C)$.

The same limit bounds the two integrals.  Because $\varphi\geq0$ we
have $|\widehat\varphi|\leq\widehat\varphi(0)=1$, so
$0\leq\Phi_{\varepsilon_1}\leq1$ and $\Phi_{\varepsilon_1}\to1$
pointwise.  Rewrite the last display as
\[
 \int(-f)\,\Phi_{\varepsilon_1}\,d\mu
 \leq f(0)-z\,\big(\widehat f(\cdot,\mathbf1)*\varphi_{\varepsilon_1}\big)(0).
\]
The integrand on the left is nonnegative, so Fatou's lemma gives
\[
 \int(-f)\,d\mu\leq f(0)-z.
\]

The Fourier side is bounded the same way.  The pairing identity reads
\[
 \int\widehat h\,d\nu=h(0)+\int h\,d\mu-z\,\widehat h(0,\mathbf1).
\]
As $\varepsilon_2\to0$ and $i\to\infty$ the right side tends to
$f(0)+\int f\Phi_{\varepsilon_1}\,d\mu
-z\big(\widehat f(\cdot,\mathbf1)*\varphi_{\varepsilon_1}\big)(0)$,
which is at most
$f(0)-z\big(\widehat f(\cdot,\mathbf1)*\varphi_{\varepsilon_1}\big)(0)$,
and $\widehat h\geq0$ tends pointwise to
$\widehat f(\cdot,\chi)*\varphi_{\varepsilon_1}$.  Fatou's lemma in
$(\varepsilon_2,i)$, and again in $\varepsilon_1$ with each channel of
$\widehat f$ continuous, gives
\[
 \int\widehat f\,d\nu\leq f(0)-z.
\]

Suppose $f(0)=z$.  The two bounds become $\int(-f)\,d\mu\leq0$ and
$\int\widehat f\,d\nu\leq0$, with nonnegative integrands against
nonnegative measures, so both integrals vanish: $f=0$ $\mu$-almost
everywhere and $\widehat f=0$ $\nu$-almost everywhere.

If in addition $\mu$ is atomic, retain a single atom of $\mu$ in the
limiting inequality: all contributions from $\mu$ are nonpositive and
$\Phi_{\varepsilon_1}\to1$ pointwise, so a strictly negative value of
$f$ at that atom contradicts equality.  Thus $f$ vanishes on
$\operatorname{supp}\mu$, and the same argument applied to a single
atom of $\nu$ on the Fourier side gives $\widehat f=0$ on
$\operatorname{supp}\nu$.
\end{proof}
\subsection{Proof of attainability}\label{app:attainability}

The argument is vague compactness of the dual feasible set together
with the vague continuity of the Fourier transform on the cone of
positive definite measures. Here $A$ may be an arbitrary compact abelian group.

\begin{lemma}[The dual feasible set is vaguely bounded]
\label{lem:app-uniform}
For every compact $K\subseteq X$ there is $c_K<\infty$ with
$\mu(K)\leq c_K$ for every $\mu$ feasible for
\eqref{eq:dual_lca_val}.
\end{lemma}

\begin{proof}
We may assume $K\cap C\neq\emptyset$, since otherwise
$\mu(K)\leq\mu(\{0\})=1$.

The mechanism is weak duality against a single test function.  Suppose
$F\in C_c(X;\R)$ is positive definite with $F\leq0$ on $C$ and
$F\leq-1$ on $K\cap C$.  A positive definite measure is nonnegative on
the continuous positive definite functions it integrates
\cite[Section~4]{AGdL}, and $F$ has compact support, so writing
$\mu=\delta_0+\kappa$ with $\kappa\geq0$ supported on $C$,
\[
 0\leq\langle\mu,F\rangle=F(0)+\int_CF\,d\kappa\leq F(0)-\kappa(K\cap C),
\]
whence $\mu(K)\leq1+F(0)$, a bound depending only on $K$.  Some such
$F$ is needed: without the constraint
$\operatorname{supp}\kappa\subseteq C$ the measures $\delta_0+t\sigma$
are feasible for every $t>0$ and every positive definite positive
$\sigma$ supported off the origin, and these are not vaguely bounded.

The sign swap lemma \cite[Lemma~4.2]{BFGRR26} supplies $F$.  Choose a
compact symmetric neighborhood $N$ of $0$ in $X$ small enough that
\[
 S:=\big((K\cap C)\cup(-(K\cap C))\big)+N
\]
avoids $0$; this is possible because $K\cap C$ is compact and
$0\notin C$.  Then $S$ is symmetric and compact, has nonempty interior
and hence positive Haar measure, and is bounded away from $0$.  Since
$C$ too is closed and avoids $0$, we may choose an open neighborhood
$V\ni0$ with compact closure so small that $W:=V-V$ satisfies both
$W\cap C=\emptyset$ and $(W+W+W)\cap S=\emptyset$.  For these data
\cite[Lemma~4.2]{BFGRR26} produces a positive definite
$F\in C_c(X;\R)$ with $F\equiv-1$ on $S$ and with $F>0$ only on $W$.
Hence $F\leq0$ off $W$, in particular on $C$, and $F\leq-1$ on
$K\cap C\subseteq S$.  That lemma also gives the explicit bound
$F(0)=2m_X(S+W)/m_X(V)$, so $\mu(K)\leq1+2m_X(S+W)/m_X(V)$.
\end{proof}
We use the following form of the vague continuity of the Fourier
transform on the cone of positive definite measures.

\begin{proposition}[{\cite[Lem.~4.11.10]{MS17}}; see also
{\cite[Proposition~8.1]{SS21}}]\label{prop:app-continuity}
Let $(\mu_\alpha)$ be a net of positive definite measures on a locally
compact abelian group converging vaguely to a measure $\mu$.  Then
$\mu$ is positive definite and $\widehat{\mu_\alpha}\to\widehat\mu$
vaguely.
\end{proposition}

The value $D_X(C)$ is finite: the function of
Lemma~\ref{lem:app-feasible} is feasible for
\eqref{eq:primal_lca_val} after dividing by
$\widehat f(0,\mathbf1)=1$, so $P_X(C)\leq f(0)<\infty$, and
Proposition~\ref{prop:app-mollify} gives $D_X(C)\leq P_X(C)$.  An
optimizing sequence therefore exists.

Let $\mu_n$ be an optimizing sequence for $D_X(C)$, and write
$z_n=\widehat{\mu_n}(\{(0,\mathbf1)\})$, so that $z_n\to D_X(C)$.
Dual feasibility asks exactly that each $\widehat{\mu_n}$ be a
nonnegative measure, so each $\mu_n$ is positive definite.  Recall that
the weak-$*$ topology on Radon measures is naturally equipped with the
vague topology.
By Lemma~\ref{lem:app-uniform} the sequence $\mu_n$ is bounded on compact sets, i.e. bounded in the vague topology.
By Banach-Alaoglu, there is some subsequential limit $\mu_n \to \mu$.

\begin{proof}[Proof of Theorem~\ref{thm:dual-attainability}]
  It suffices to show that the limiting measure $\mu$ is feasible and optimal for $D_X(C)$.

  The support condition is checked by noting that
  \[\operatorname{supp} \mu \subseteq \bigcup_n \operatorname{supp} \mu_n \subseteq \{ 0 \} \cup C.\]
  Similarly, by considering a test function around $0$, we find $\mu(\{0\})$ is a limit point of $\mu_n(\{0\}) = 1$.

  Proposition~\ref{prop:app-continuity} supplies the remaining
  condition $\widehat\mu\geq0$, and with it the convergence
  $\widehat{\mu_n}\to\widehat\mu$ in the vague topology.  Therefore
  $\mu$ is feasible, and it remains to compute the objective.

  For a vaguely convergent net $\rho_\alpha\to\rho$ of nonnegative
  measures and a compact set $K$ one has
  $\limsup_\alpha\rho_\alpha(K)\leq\rho(K)$: any $f\in C_c$ with
  $f\geq\mathbf 1_K$ gives
  $\limsup_\alpha\rho_\alpha(K)\leq\lim_\alpha\int f\,d\rho_\alpha
  =\int f\,d\rho$, and the infimum of the right-hand side over such $f$
  is $\rho(K)$ by outer regularity.  Applying this to
  $\widehat{\mu_n}\to\widehat\mu$ at the compact set
  $\{(0,\mathbf1)\}$,
  \[
   D_X(C)=\lim_nz_n\leq\widehat\mu\big(\{(0,\mathbf1)\}\big).
  \]
  The reverse inequality holds because $\mu$ is dual feasible, so its
  objective value equals $D_X(C)$ and the supremum is attained.
\end{proof}

\subsection{Primal attainment}\label{app:primal-attainment}

The argument below adapts \cite[Proposition~A.1]{CdLS22}, which in turn
follows \cite[Theorem~3]{GOS17}, to the setting of $X=\R^k\times A$.

\begin{proof}[Proof of Theorem~\ref{thm:lca-primal-attainment}]
Since $U$ is a neighborhood of the identity, $C=X\setminus U$ is closed
and $0\notin C$, so Lemma~\ref{lem:app-feasible} supplies a function
feasible for \eqref{eq:primal_lca_val}; in particular the feasible set
is nonempty and $P_X(C)<\infty$.  Normalize the mutually dual Haar
measures on $X=\R^k\times A$ and
$\widehat X=\R^k\times\widehat A$ so that Plancherel and Fourier inversion hold.
Recall that for locally compact abelian groups $X$, if
$g\in C_0(X)\cap L^1(X)$ is positive definite, then its Fourier transform
$\widehat g$ is a nonnegative integrable function, and
\[
 g(0)=\int_{\widehat X}\widehat g\,dm_{\widehat X},
 \qquad
 \widehat g(0)=\int_Xg\,dm_X.
\]

Choose feasible $f_j$ with $f_j(0)\to P_X(C)$.  Positive definiteness gives
$|f_j(x)|\leq f_j(0)$ for every $x\in X$.  Since $f_j\leq0$ on $C$ and
$\int_Xf_j=1$,
\begin{align*}
 \|f_j\|_1
 &=\int_U(|f_j|+f_j)\,dm_X-1
 \leq 2m_X(U)f_j(0)-1,\\
 \|f_j\|_2^2&\leq\|f_j\|_\infty\|f_j\|_1.
\end{align*}
In particular every feasible function has
$f_j(0)\geq m_X(U)^{-1}$, so $P_X(C)>0$.
Thus $(f_j)$ is bounded in $L^2(X)$.  After passing to a subsequence it
converges weakly in $L^2$; by Mazur's lemma, tail convex combinations
converge strongly.  All the constraints are convex and the objectives of
these combinations still tend to $P_X(C)$, so we relabel and suppose that
$f_j\to f$ strongly in $L^2(X)$.  Plancherel gives
$\widehat f_j\to h:=\mathcal Ff$ strongly in $L^2(\widehat X)$.  Passing to
a further subsequence, both convergences hold almost everywhere.

We have $h\geq0$ almost everywhere, and Fatou's lemma yields
\[
 \int_{\widehat X}h\,dm_{\widehat X}
 \leq\liminf_j\int_{\widehat X}\widehat f_j\,dm_{\widehat X}
 =P_X(C).
\]
Hence $h\in L^1(\widehat X)$, and its inverse Fourier transform is a
continuous function in $C_0(X)$.  It agrees almost everywhere with the
$L^2$ limit $f$, and from now on $f$ denotes this continuous
representative.  In particular
$\|f\|_\infty\leq\|h\|_1$, so $f$ is integrable on $U$.  On $C$ we have
$f\leq0$ almost everywhere, and Fatou applied to $-f_j$ on $C$ shows that
$f$ is integrable there as well.  Thus $f\in C_0(X)\cap L^1(X)$.

The condition $C=\overline{\operatorname{int}C}$ now promotes the sign
constraint to a pointwise one: the almost-everywhere inequality gives
$f\leq0$ on the open set $\operatorname{int}C$, by continuity and the fact
that Haar measure is positive on nonempty open sets, and then on its
closure.  Since $f\in L^1(X)$, its Fourier transform has a continuous
representative.  This representative agrees almost everywhere with $h$,
so it is nonnegative everywhere.  Consequently $f$ is positive definite.

It remains to recover the normalization.  The uniform bound on $U$ and
almost-everywhere convergence give
$\int_Uf_j\to\int_Uf$ by dominated convergence.  Fatou on the nonnegative
functions $-f_j|_C$ gives
\[
 \int_Cf\,dm_X\geq\limsup_j\int_Cf_j\,dm_X
 =1-\int_Uf\,dm_X,
\]
and therefore $b:=\int_Xf\,dm_X\geq1$.  The function $f/b$ is primal
feasible and
\[
 P_X(C)\leq \frac{f(0)}b\leq f(0)
 =\int_{\widehat X}h\,dm_{\widehat X}\leq P_X(C).
\]
All inequalities are equalities, so $f/b$ attains the primal infimum (and,
in fact, $b=1$).
\end{proof}

\section*{Data Availability}

The certificates underlying the $D_5/A_3$ result of
Section~\ref{sec:reduction}, together with a verification script and its
test, are distributed as ancillary files with the arXiv version of this
paper.

\end{document}